\documentclass[a4paper, 11pt]{article}
\usepackage{amsmath}
\usepackage{amssymb,esint}
\usepackage{amscd}
\usepackage{xspace}
\usepackage{fancyhdr}
\usepackage{color}
\usepackage{authblk}
\usepackage{srcltx} 
\newtheorem{theorem}{Theorem}[section]

\newtheorem{definition}[theorem]{Definition}

\newtheorem{lemma}[theorem]{Lemma}

\newtheorem{remark}[theorem]{Remark}

\newenvironment{proof}[1][Proof]{\textbf{#1.} }{\hfill\rule{0.5em}{0.5em}}
{\catcode`\@=11\global\let\AddToReset=\@addtoreset
\AddToReset{equation}{section}

\AddToReset{theorem}{section}

\def\nc{\newcommand}

\def\div{\text{div}}

\nc\pa{\partial}

\nc\CC{\mathbb{C}}
\nc\RR{\mathbb{R}}
\nc\QQ{\mathbb{Q}}
\nc\ZZ{\mathbb{Z}}
\nc\NN{\mathbb{N}}

\begin{document}
\title{Lorentz-Morrey global bounds for singular quasilinear elliptic equations with measure data}

\author{M.-P. Tran\thanks{Applied Analysis Research Group, Faculty of Mathematics and Statistics, Ton Duc Thang University, Ho Chi Minh city, Vietnam; \texttt{tranminhphuong@tdtu.edu.vn}} , T.-N. Nguyen\thanks{Department of Mathematics, Ho Chi Minh City University of Education, Ho Chi Minh city, Vietnam}}

\date{}  
\maketitle
\begin{abstract}

The aim of this paper is to present the global estimate for gradient of renormalized solutions to the following quasilinear elliptic problem:
\begin{align*}
\begin{cases}
-\div(A(x,\nabla u)) &= \mu \quad \text{in} \ \ \Omega, \\
u &=0 \quad \text{on} \ \ \partial \Omega,
\end{cases}
\end{align*}
in Lorentz-Morrey spaces, where $\Omega \subset \mathbb{R}^n$ ($n \ge 2$), $\mu$ is a finite Radon measure,  $A$ is a monotone Carath\'eodory vector valued function defined on $W^{1,p}_0(\Omega)$ and the $p$-capacity uniform thickness condition is imposed on the complement of our domain $\Omega$. It is remarkable that the local gradient estimates has been proved firstly by G. Mingione in \cite{Mi3} at least for the case $2 \le p \le n$, where the idea for extending such result to global ones was also proposed in the same paper. Later, the global Lorentz-Morrey and Morrey regularities were obtained by N.C.Phuc in \cite{55Ph1} for regular case $p>2 - \frac{1}{n}$. Here in this study, we particularly restrict ourselves to the singular case $\frac{3n-2}{2n-1}<p\le 2-\frac{1}{n}$. The results are central to generalize our technique of good-$\lambda$ type bounds in previous work \cite{MP2018}, where the local gradient estimates of solution to this type of equation was obtained in the Lorentz spaces. Moreover, the proofs of most results in this paper are formulated globally up to the boundary results.

\medskip

\medskip

\medskip

\noindent 

\medskip

\noindent {\bf Keywords:} quasilinear elliptic equation; measure data; Lorentz-Morrey space; capacity uniformly thickness; global bounds, gradient estimates.

\end{abstract}   
                  
\section{Introduction}
\label{sec:intro}

Our main purpose in this paper is to establish a global gradient estimate in Lorentz-Morrey spaces of solutions (the renormalized solutions) to the following quasilinear elliptic equations with respect to the given measure datum $\mu$:
\begin{equation}
\label{eq:elliptictype}
\begin{cases}
-\div(A(x,\nabla u)) &= \mu \quad \text{in} \ \ \Omega, \\
u &=0 \quad \text{on} \ \ \partial \Omega.
\end{cases}
\end{equation}

In our study, the given domain $\Omega$ is a bounded open subset of $\mathbb{R}^n$, $n \ge 2$, and $\mu$ stands for a finite signed Radon measure in $\Omega$. The nonlinear operator $A: \mathbb{R}^n \times \mathbb{R}^n \rightarrow \mathbb{R}$ is a Carath\'eodory vector valued function (that is, $A(.,\xi)$ is measurable on $\Omega$ for every $\xi$ in $\mathbb{R}^n$, and $A(x,.)$ is continuous on $\mathbb{R}^n$ for almost every $x$ in $\Omega$) which satisfies the following growth and monotonicity conditions: for some $1<p\le n$:
\begin{align}
\label{eq:A1}
\left| A(x,\xi) \right| &\le \beta |\xi|^{p-1},
\end{align}
\begin{align}
\label{eq:A2}
\langle A(x,\xi)-A(x,\eta), \xi - \eta \rangle &\ge \alpha \left( |\xi|^2 + |\eta|^2 \right)^{\frac{p-2}{2}}|\xi - \eta|^2,
\end{align}
for every $(\xi,\eta) \in \mathbb{R}^n \times \mathbb{R}^n \setminus \{(0,0)\}$ and a.e. $x \in \mathbb{R}^n$, $\alpha$ and $\beta$ are positive constants.

In addition, in order to obtain the global bounds of solution in Lorentz-Morrey spaces, the domain $\Omega \subset \mathbb{R}^n$ is under the assumption that its complement $\mathbb{R}^n \setminus \Omega$ is uniformly $p$-capacity thick. More precise, we say that the domain $\mathbb{R}^n \setminus \Omega$ satisfies the $p$-capacity uniform thickness condition if there exist two constants $c_0,r_0>0$ such that
\begin{align}
\label{eq:capuni}
\text{cap}_p((\mathbb{R}^n \setminus \Omega) \cap \overline{B}_r(x), B_{2r}(x)) \ge c_0 \text{cap}_p(\overline{B}_r(x),B_{2r}(x)),
\end{align}
for every $x \in \mathbb{R}^n \setminus \Omega$ and $0<r \le r_0$. Here, the $p$-capacity of any compact set $K \subset \Omega$ is defined as:
\begin{align*}
\text{cap}_p(K,\Omega) = \inf \left\{ \int_\Omega{|\nabla \varphi|^p dx}: \varphi \in C_c^\infty, \varphi \ge \chi_K \right\},
\end{align*}
where $\chi_K$ is the characteristic function of $K$. This $p$-capacity density condition is stronger than the Weiner criterion in~\cite{Kilp}:
\begin{align*}
\int_0^1{\left(\frac{\text{cap}_p((\mathbb{R}^n \setminus \Omega) \cap \overline{B}_r(x), B_{2r}(x))}{\text{cap}_p(\overline{B}_r(x), B_{2r}(x))} \right)^{\frac{1}{p-1}} \frac{dr}{r}} = \infty
\end{align*} 
which characterizes regular boundary points for the Dirichlet problem for the $p$-Laplace equation. Otherwise, 
it is weaker than the Reifenberg flatness condition that was discussed in various studies \cite{BP14, BR, BW1, BW2, MP11, MP12,  ER60}. The class of domains whose complement satisfies the uniformly $p$-capacity condition is relatively large (including those with Lipschitz boundaries or satisfy a uniform corkscrew condition). The condition \eqref{eq:capuni} is still valid for balls centered outside a uniformly $p$-thick domain  and furthermore, this condition is nontrivial when $p \le n$. The definition and properties of variational capacity can be found in \cite{M2005}.

Throughout this paper, the solution to the problem \eqref{eq:elliptictype} is considered in the sense of \emph{renormalized solution}, whose definition was presented in \cite{BMMP, BGO, 11DMOP} and many references therein. More specifically, the datum measure $\mu$ is defined in $\mathfrak{M}_b(\Omega)$, the space of all Radon measures on $\Omega$ with bounded total variation. Note that, if $\mu \in \mathfrak{M}_b(\Omega)$, then the total variation of $\mu$ is bounded positive measure on $\Omega$. It is also remarkable that for every measure $\mu$ in $\mathfrak{M}_b(\Omega)$ there exists a unique pair of measures $(\mu_0, \mu_s)$, with $\mu_0$ in $\mathfrak{M}_b(\Omega)$ and $\mu_s$ in $\mathfrak{M}_s(\Omega)$, such that $\mu = \mu_0+\mu_s$, is $\mu$ is nonnegative, so are $\mu_0$ and $\mu_s$. Therefore, the measures $\mu_0$ and $\mu_s$ will be called the \emph{absolutely continuous} and the \emph{singular} part of $\mu$ with respect to the $p$-capacity. 

The quasilinear elliptic equations with measure data \eqref{eq:elliptictype} and solution regularity estimates have been widely studied in several papers in recent years. For instance, firstly by L. Boccardo et al. in \cite{BGO}, and later in different works by G. Mingione et al. \cite{KMi1, KMi2, Mi2, Mi3, Mi4} and N.C. Phuc et al. \cite{AdP1, AdP2, 55QH4, 55Ph2, 55Ph0, 55Ph1}. In \cite{Mi3}, G. Mingione firstly proposed the local estimates of solution at least for the case $2 \le p \le n$, and the extension to global estimates has also been mentioned by using maximal function. Later, some of other researching approaches have been studied for different hypotheses of domain $\Omega$, the nonlinear operator $A$ and the case of $p$. In \cite{55Ph0}, N.C. Phuc gave the global gradient estimates in the Lorentz spaces and later in \cite{55Ph1}, author also presented his study on the Lorentz-Morrey and Morrey global bounds to this type of equation, for the regular case of $2-\frac{1}{n}<p \le n$ and $\Omega$ is subject to the $p$-capacity complement thickness condition.  There have been further discussions on the global gradient estimates of solution to this equation, with different possible assumptions. For instance, authors in \cite{55QH4} studied the gradient estimate of solution in Lorentz space under the hypotheses of $\Omega$-Reifenberg domain, for $\frac{3n-2}{2n-1}<p\le 2-\frac{1}{n}$ and the nonlinearity $A$ is required to satisfy the smallness condition of BMO type. Otherwise, without the assumption of Reifenberg flat domain, the gradient estimates were presented under the weaker condition on complement domain of $\Omega$, that is the $p$-capacity uniform thickness ($p$-fat). Later, in our present work, for singular case $\frac{3n-2}{2n-1}<p\le 2-\frac{1}{n}$, the solution regularity to this quasilinear elliptic equation \eqref{eq:elliptictype} in Lorentz spaces $L^{q,s}(\Omega)$ ($q>0, 0<s \le \infty$) were given in \cite{MP2018}.

In the present paper, our work is studied following the series of works by G. Mingione (in \cite{55DuzaMing}, \cite{Duzamin2}, \cite{KMi1, KMi2}, \cite{Mi2,Mi3}), N.C. Phuc (in \cite{AdP1, 55Ph2, 55Ph0, 55Ph1}), where the global bounds of solution to \eqref{eq:elliptictype} were obtained under different hypotheses and assumptions. Herein, our main advantage here is to provide a new continuation result of solution global bounds in Lorentz-Morrey spaces for singular $p$, where the proof techniques may be generalized in the same way as our previous work in \cite{MP2018}. However, the main new contribution in this paper is that all gradient estimates are \emph{global} up to the boundary. This research paper gives us a motivation to study global $W^{\alpha,p}$ estimates ($0<\alpha<1$), that will be the subject of a forthcoming paper.

The plan of this paper is organized as follows. We present in the next section some backgrounds and main results via some theorems are stated therein. Section \ref{sec:lems} deals with some local and global comparison estimates, including interior and boundary estimates of solution. And finally in Section \ref{sec:main} we indicate how our techniques may be used to give the proofs of desired results.

\section{Main Theorems}
\label{sec:pre}
Let us firstly recall the definition of the \emph{Lorentz space} $L^{q,t}(\Omega)$ for $0<q<\infty$ and $0<t\le \infty$ (see in \cite{55Gra}). It is the set of all Lebesgue measurable functions $g$ on $\Omega$ such that:
\begin{align}
\label{eq:lorentz}
\|g\|_{L^{q,t}(\Omega)} = \left[ q \int_0^\infty{ \lambda^q\mathcal{L}^n \left( \{x \in \Omega: |g(x)|>\lambda\} \right)^{\frac{t}{q}} \frac{d\lambda}{\lambda}} \right]^{\frac{1}{t}} < +\infty,
\end{align}
as $t \neq \infty$. If $t = \infty$, the space $L^{q,t}(\Omega)$ is the usual weak-$L^q$ or Marcinkiewicz space with the following quasinorm:
\begin{align}
\|g\|_{L^{q,\infty}(\Omega)} = \sup_{\lambda>0}{\lambda \mathcal{L}^n\left(\{x \in \Omega:|g(x)|>\lambda\}\right)^{\frac{1}{q}}}.
\end{align}
In the definition above, the notation $\mathcal{L}^n(E)$ stands for the $n$-dimensional Lebesgue measure of a set $E \subset \mathbb{R}^n$. When $t=q$, the Lorentz space $L^{q,q}(\Omega)$ becomes the Lebesgue space $L^q(\Omega)$. 

Otherwise, we also give the definition of \emph{Lorentz-Morrey spaces}. A function $g \in L^{q,t}(\Omega)$ for $0<q<\infty$, $0<t \le \infty$ is said to belong to the Lorentz-Morrey functional spaces $L^{q,t;\kappa}(\Omega)$ for some $0<\kappa \le n$  if
\begin{align}
\label{eq:LMsp}
\|g\|_{L^{q,t;\kappa}(\Omega)}:=\sup_{0<\rho<diam(\Omega); x \in \Omega}{\rho^{\frac{\kappa-n}{q}}}\|g\|_{L^{q,t}(B_\rho(x)\cap\Omega)} < +\infty.
\end{align}
When $\kappa = n$ the space $L^{q,t;\kappa}(\Omega)$ is exactly the space $L^{q,t}(\Omega)$.

In this part, let us also recall and reproduce the definition of \emph{renormalized solution}, that was given in \cite{BMMP, BGO, 11DMOP} and currently in our previous paper \cite[Section 2.4]{MP2018}. 

\begin{definition}
\label{def:renormsol3}
Let $\mu = \mu_0+\mu_s \in \mathfrak{M}_b(\Omega)$, where $\mu_0 \in \mathfrak{M}_0(\Omega)$ and $\mu_s \in \mathfrak{M}_s(\Omega)$. A measurable function $u$ defined in $\Omega$ and finite almost everywhere is called a renormalized solution of \eqref{eq:elliptictype} if $T_k(u) \in W^{1,p}_0(\Omega)$ for any $k>0$, $|{\nabla u}|^{p-1}\in L^r(\Omega)$ for any $0<r<\frac{n}{n-1}$, and $u$ has the following additional property. For any $k>0$ there exist  nonnegative Radon measures $\lambda_k^+, \lambda_k^- \in\mathfrak{M}_0(\Omega)$ concentrated on the sets $u=k$ and $u=-k$, respectively, such that $\mu_k^+\rightarrow\mu_s^+$, $\mu_k^-\rightarrow\mu_s^-$ in the narrow topology of measures and  that
 \begin{align*}
 \int_{\{|u|<k\}}\langle A(x,\nabla u),\nabla \varphi\rangle
  	dx=\int_{\{|u|<k\}}{\varphi d}{\mu_{0}}+\int_{\Omega}\varphi d\lambda_{k}%
  	^{+}-\int_{\Omega}\varphi d\lambda_{k}^{-},
 \end{align*}
  	for every $\varphi\in W^{1,p}_0(\Omega)\cap L^{\infty}(\Omega)$.
\end{definition}
In definition \ref{def:renormsol3}, we follow the notation of operator $T_k$ defined as:
\begin{align}
\label{eq:Tk}
T_k (s) = \max\left\{ -k,\min\{k,s\} \right\}, \quad k \in \mathbb{R}^+, \ s \in \mathbb{R},
\end{align}
that belongs to $W_0^{1,p}(\Omega)$ for every integer $k>0$, which satisfies $-\div A(x,\nabla T_k(u)) = \mu_k$ in the sense of distribution in $\Omega$ for a finite measure $\mu_k$ in $\Omega$. 
\begin{definition}
\label{def:truncature}
Let $u$ be a measurable function defined on $\Omega$ which is finite almost everywhere, and satisfies $T_k(u) \in W^{1,1}_0(\Omega)$ for every $k>0$. Then, there exists a unique measurable function $v: \Omega \to \mathbb{R}^n$ such that:
\begin{align}
\nabla T_k(u) = \chi_{\{|u| \le k\}} v , \quad \text{almost everywhere in} \ \ \Omega, \quad \text{for  every} \ k>0.
\end{align}
Moreover, the function $v$ is so-called ``distributional gradient $\nabla u$'' of $u$.
\end{definition}

The following Remark characterizes the gradient estimate for solution $u$ to \eqref{eq:elliptictype}. For the proof, we refer the reader to \cite[Theorem 4.1]{11DMOP}.
\begin{remark}
	\label{rem:nablau}
	Let $\Omega$ is an open bounded domain in $\mathbb{R}^n$. Then, there exists $C=C(n,p,\alpha, \beta)$ such that for any the renormalized solution $u$ to \eqref{eq:elliptictype} with a given finite measure data $\mu$ there holds:
	\begin{align}
	\label{eq:nablau}
	\|\nabla u\|_{L^{\frac{(p-1)n}{n-p},\infty}(\Omega)}\leq C\left[|\mu|(\Omega)\right]^{\frac{1}{p-1}}.
	\end{align}
\end{remark}

Let us now state main results of boundedness property of maximal function and gradient estimates of solution to \eqref{eq:elliptictype} on Lorentz-Morrey spaces, where the proofs would be found in Section \ref{sec:main}, respectively.

\begin{theorem}
	\label{theo:lambda_estimate}
	 Let $\frac{3n-2}{2n-1}<p \le 2-\frac{1}{n}$ and suppose that $\Omega \subset \mathbb{R}^n$ is a bounded domain whose complement satisfies a $p$-capacity uniform thickness condition with constants $c_0, r_0>0$. Let $\mu \in \mathfrak{M}_b(\Omega)$, $0<R_0<diam(\Omega)$ and the balls $D_1=B_{R_0}(x_0), D_2=B_{20R_0}(x_0)$ be such that $D_1\cap\Omega\not =\emptyset$, where $x_0$ is fixed in $\Omega$. 
	Then, for any $\gamma_0 \in \left[\frac{2-p}{2},\frac{(p-1)n}{n-1} \right)$ and for any renormalized solution $u$ to \eqref{eq:elliptictype} with given measure data $\mu$, there exist $\Theta = \Theta(n,p,\alpha,\beta,c_0)>p$ and constant $C>0$ depending on $n,p,\alpha,\beta,c_0,diam(\Omega)/{r_0}$ such that the following estimate
	\begin{align}
		\label{eq:mainlambda}
		&\mathcal{L}^n\left(\left\{({\bf M}(\chi_{D_2}|\nabla u|^{\gamma_0}))^{1/\gamma_0}>\varepsilon^{-\frac{1}{\Theta}}\lambda, (\mathbf{M}_1(\chi_{D_2}\mu))^{\frac{1}{p-1}}\le \varepsilon^{\frac{1}{(p-1)\gamma_0}}\lambda \right\}\cap D_1 \right) \nonumber\\
		&\qquad\leq C \varepsilon  \mathcal{L}^n \left(\left\{ ({\bf M}(\chi_{D_2}|\nabla u|^{\gamma_0}))^{1/\gamma_0}> \lambda\right\}\cap D_1\right)
	\end{align}
	holds for any $\lambda>\varepsilon^{-\frac{1}{(p-1)\gamma_0}}\|\nabla u\|_{L^{\gamma_0}(D_2)}R_0^{-\frac{n}{\gamma_0} }, \varepsilon\in (0,1)$.
\end{theorem}

\begin{theorem}
\label{theo:bst} 
 Let $\frac{3n-2}{2n-1}<p \le 2-\frac{1}{n}$ and suppose that $\Omega \subset \mathbb{R}^n$ is a bounded domain whose complement satisfies a $p$-capacity uniform thickness condition with constants $c_0, r_0>0$. Then, there exist $\Theta = \Theta(n,p,\alpha,\beta,c_0)>p$, $\beta_0 = \beta_0(n,p,\alpha,\beta) \in (0,\frac{1}{2}]$ and a constant $C = C(n,p,\alpha,\beta, c_0,diam(\Omega)/r_0)>0$ such that for any $0<q<\Theta$, $0<s \le \infty$, $1+(p-1)(1-\beta_0) < \theta \le n$ and for any solution $u$ to \eqref{eq:elliptictype} with a finite measure $\mu \in \mathfrak{M}_b(\Omega)$, there holds
	\begin{align*}
&	\sup_{\rho \in (0,T_0),x_0 \in \Omega}\rho^{-\frac{n}{q}+\frac{\theta-1}{p-1}} \|({\bf M}(\chi_{B_{10\rho}(x_0)}|\nabla u|^{\gamma_0}))^{1/\gamma_0}\|_{L^{q,s}(B_\rho(x_0))}
\\&\leq  C \|\mathbf{M}_{\theta}^{T_0}(|\mu|)\|_{L^\infty(\Omega)}^{\frac{1}{p-1}}+ \sup_{\rho \in (0,T_0),x_0 \in \Omega}C\rho^{-\frac{n}{q}+\frac{\theta-1}{p-1}}\|(\mathbf{M}_1(\chi_{B_{10\rho}}\mu))^{\frac{1}{p-1}}\|_{L^{q,s}(B_\rho(x_0))},
\end{align*} 
for any $\gamma_0 \in \left[\frac{2-p}{2}, \frac{(p-1)n}{n-1}\right)$, where $T_0 = diam(\Omega)$.
\end{theorem}

It can be noticed that in this theorem and in what follows, for simplicity, the set $\{x \in \Omega: |g(x)| > \Lambda\}$ is denoted by $\{|g|>\Lambda\}$ (in order to avoid the confusion that may arise). And the fractional  maximal function $\mathbf{M}_\alpha$ of each locally finite measure $\mu$ by:
\begin{align}
\label{eq:Malpha}
\mathbf{M}_\alpha(\mu)(x) = \sup_{\rho>0}{\frac{|\mu|(B_\rho(x))}{\rho^{n-\alpha}}}, \quad \forall x \in 
\mathbb{R}^n, ~0<\alpha<n.
\end{align}
For the case $\alpha=0$, the definition of $\mathbf{M}_\alpha$ becomes $\mathbf{M}_0$ is essentially the Hardy-Littlewood maximal function $\mathbf{M}$ defined for each locally integrable function $f$ in $\mathbb{R}^n$ by:
\begin{align}
\label{eq:M0}
\mathbf{M}(f)(x) = \sup_{\rho>0}{\fint_{B_{\rho}(x)}|f(y)|dy},~~ \forall x \in \mathbb{R}^n.
\end{align}
Otherwise, the notion of fractional maximal function $\mathbf{M}_\alpha^T$ is also defined as:
\begin{align}
\label{eq:MaT}
\mathbf{M}_{\alpha}^{T}(|\mu|)(y)=\sup_{0<\rho<T}\frac{|\mu|(B_\rho(y))}{\rho^{n-\alpha}},
\end{align}
for any $T>0$ and $0<\alpha<n$.

\begin{remark}
\label{rem:boundM}
It refers to \cite{55Gra} that the operator $\mathbf{M}$ is bounded from $L^s(\mathbb{R}^n)$ to $L^{s,\infty}(\mathbb{R}^n)$, for $s \ge 1$, this means,
\begin{align}
\mathcal{L}^n\left(\{\mathbf{M}(g)>\lambda\}\right) \le \frac{C}{\lambda^s}\int_{\mathbb{R}^n}{|g|^s dx}, \quad \mbox{ for all } \lambda>0.
\end{align}
\end{remark}
\begin{remark}
\label{rem:boundMlorentz}
In \cite{55Gra}, it allows us to present a boundedness property of maximal function $\mathbf{M}$ in the Lorentz space $L^{q,s}(\mathbb{R}^n)$, for $q>1$ as follows:
\begin{align}
\label{eq:boundM}
\|\mathbf{M}(f)\|_{L^{q,s}(\Omega)} \le C \|f\|_{L^{q,s}(\Omega)}.
\end{align}
\end{remark}

The following theorem provides our main result of gradient estimate of solution in the Lorentz-Morrey spaces.

\begin{theorem}
\label{theo:P}
 Let $\frac{3n-2}{2n-1}<p \le 2-\frac{1}{n}$, $\mu \in \mathfrak{M}_b(\Omega)$ and suppose that $\Omega \subset \mathbb{R}^n$ is a bounded domain whose complement satisfies a $p$-capacity uniform thickness condition with constants $c_0, r_0>0$. Then, there exist $\Theta = \Theta(n,p,\alpha,\beta,c_0)>p$, $\beta_0 = \beta_0(n,p,\alpha,\beta) \in (0,1/2]$ and $C = C(n,p,\alpha,\beta, c_0,diam(\Omega)/r_0)>0$ such that for any $0<q<\Theta$, $0<s \le \infty$, $1+(p-1)(1-\beta_0) < \theta \le n$ and for any solution $u$ to \eqref{eq:elliptictype} with a finite measure $\mu \in L^{\frac{q(\theta-1)}{\theta(p-1)}, \frac{s(\theta-1)}{\theta(p-1)};\frac{q(\theta-1)}{p-1}}(\Omega)$ there holds
\begin{align}\label{eq:globalbound}
&\|\nabla u\|_{L^{q,s;\frac{q(\theta-1)}{p-1}}(\Omega)} \leq C \||\mu|^{\frac{1}{p-1}}\|_{L^{\frac{q(\theta-1)}{\theta}, \frac{s(\theta-1)}{\theta};\frac{q(\theta-1)}{p-1}}(\Omega)}.
\end{align}
\end{theorem}

\begin{remark}
\label{rem:mingione}
In this work, it remarks that at least for the case $2 \le p \le n$, a local bounds of \eqref{eq:globalbound} was studied by G. Mingione in \cite{Mi3}.
\end{remark}

\section{Comparison estimates}
\label{sec:lems}

This section is intended to obtain the local interior and boundary comparison estimates that are essential to our development later. 

In a certain range of singular $p$, we always suppose that the domain $\Omega \subset \mathbb{R}^n$ is a bounded domain whose complement satisfies a $p$-capacity uniform thickness condition with constants $c_0, r_0>0$. And for simplicity of notation, the constant $C$ we mention in what follows always depends on some given constants $n,p$ and $\alpha, \beta>0$ of the Carath\'eodory vector valued function $\mathcal{A}$.

\subsection{Interior Estimates}
\label{sec:interior}
First, we will take our attention to the interior estimates. Let us fix a point $x_0 \in \Omega$, for $0<2R \le r_0$ ($r_0$ was given in \eqref{eq:capuni}) and $\mu \in \mathfrak{M}_b(\Omega)$. Assume $u\in W_{0}^{1,p}(\Omega)$ being solution to \eqref{eq:elliptictype} and for each ball $B_{2R}=B_{2R}(x_0)\subset\subset\Omega$, we consider the unique solution $w\in W_{0}^{1,p}(B_{2R})+u$ to the following equation:
\begin{equation}
\label{111120146}
\left\{ \begin{array}{rcl}
- \operatorname{div}\left( {A(x,\nabla w)} \right) &=& 0 \quad \text{in} \quad B_{2R}, \\ 
w &=& u\quad \text{on} \quad \partial B_{2R}.  
\end{array} \right.
\end{equation}

In this section, we are going to deal with some basic estimates of renormalized solution $u$ to \eqref{eq:elliptictype} in comparison to the solution $w$ to \eqref{111120146}. For the convenience of the reader, we repeat the results in relevant materials, via Lemmas \ref{111120147}, \ref{lem:estimateinter}, \ref{lem:res3.3P} and \ref{lem:res3.4P} herein without proofs, then making our proof of Lemma \ref{lem:res7.2H}.

We first recall the following version of interior Gehring's lemma applied to the function $w$ defined in equation \eqref{111120146}, has been studied in \cite[Theorem 6.7]{Giu}. It is also known as a kind of ``reverse'' H\"older inequality with increasing supports.

\begin{lemma} \label{111120147} 
Let $w$ be the solution to \eqref{111120146}. Then, there exist  constants $\Theta = \Theta(n,p,\alpha, \beta)>p$ and $C = C(n,p,\alpha,\beta)>0$ such that the following estimate      
	\begin{equation}\label{111120148}
	\left(\fint_{B_{\rho/2}(y)}|\nabla w|^{\Theta} dx\right)^{\frac{1}{\Theta}}\leq C\left(\fint_{B_{\rho}(y)}|\nabla w|^{p-1} dx\right)^{\frac{1}{p-1}}
	\end{equation}
	holds for all  $B_{\rho}(y)\subset B_{2R}(x_0)$. 
\end{lemma} 

The next lemma gives an estimate for the difference $\nabla u-\nabla w$. These results were described and proved in \cite[Lemma 2.2, 2.3]{55QH4}. 
\begin{lemma}
\label{lem:estimateinter}
Let $w$ be solution to \eqref{111120146}. Then, for any $\frac{2-p}{2}\leq \gamma_0<\frac{(p-1)n}{n-1}\leq 1$, there is a constant  $C=C(n,p,\alpha,\beta)>0$ such that:
	\begin{align}
	\label{eq:estimateinter}
	\begin{split}
	\left(	\fint_{B_{2R}(x_0)}|\nabla (u-w)|^{\gamma_0}dx\right)^{\frac{1}{\gamma_0}}&\leq C\left[ \frac{|\mu|(B_{2R}(x_0))}{R^{n-1}} \right]^{\frac{1}{p-1}}\\
	&\qquad +	C \frac{|\mu|(B_{2R}(x_0))}{R^{n-1}} \left(	\fint_{B_{2R}(x_0)}|\nabla u|^{\gamma_0}dx\right)^{\frac{2-p}{\gamma_0}}.
	\end{split}
	\end{align}
\end{lemma}

The following lemma comes from the standard interior H\"older continuity of solutions, that can be found in \cite[Theorem 7.7]{Giu}.

\begin{lemma}
\label{lem:res3.3P}
Let $w$ be solution to \eqref{111120146}. Then, there exists a constant $\beta_0 = \beta_0(n,p,\alpha,\beta) \in (0,1/2]$ such that:
\begin{align*}
\left(\fint_{B_\rho(y)}{|w - \overline{w}_{B_\rho(y)}|^p dx} \right)^{\frac{1}{p}} \le C\left(\frac{\rho}{r} \right)^{\beta_0} \left( \fint_{B_r(y)}{|w - \overline{w}_{B_r(y)}|^p dx} \right)^{\frac{1}{p}},
\end{align*}
for any $y \in B_{2R}(x_0)$ with $B_\rho(y) \subset B_r(y) \subset B_{2R}(x_0)$. Moreover, there exists a constant $C=C(n,p,\alpha,\beta)>0$ such that we have the following estimate:
\begin{align}
\label{eq:nablaw_estimate}
\left(\fint_{B_\rho(y)}{|\nabla w|^p dx} \right)^{\frac{1}{p}} \le C \left(\frac{\rho}{r} \right)^{\beta_0-1}\left(\fint_{B_r(y)}{|\nabla w|^p dx} \right)^{\frac{1}{p}},
\end{align}
holds for any $y \in B_{2R}(x_0)$ such that $B_\rho(y) \subset B_r(y) \subset B_{2R}(x_0)$.
\end{lemma}

It can be noticed that the denotation $\overline{w}_{B_\rho(y)}$ indicates the average integral of $w$ over the ball $B_\rho(y)$. Applying Lemma \ref{lem:estimateinter}, the inequality \eqref{eq:nablaw_estimate} can be further improved as in the following lemma.

\begin{lemma}
\label{lem:res3.4P}
Let $w$ be solution to \eqref{111120146}. Then, for any $\Theta \in (0,p]$, there exist constants $\beta_0 = \beta_0(n,p,\alpha,\beta)  \in (0,1/2]$, $C = C(n,p,\alpha,\beta,\Theta)>0$, there holds
\begin{align*}
\left( \fint_{B_\rho(y)}{|\nabla w|^\Theta dx} \right)^{\frac{1}{\Theta}} \le C \left(\frac{\rho}{r} \right)^{\beta_0-1} \left(\fint_{B_r(y)}{|\nabla w|^\Theta dx} \right)^{\frac{1}{\Theta}},
\end{align*}
for any $y \in B_{2R}(x_0)$ such that $B_\rho(y) \subset B_r(y) \subset B_{2R}(x_0)$.
\end{lemma}

\begin{lemma}
\label{lem:res7.2H}
Let $\beta_0 \in (0,1/2]$ be as in Lemmas \ref{lem:res3.3P} and \ref{lem:res3.4P}. Then, for any $\delta \in \left[-\frac{n-p}{p-1},\beta_0\right)$, there exists a constant $C = C(n,p,\alpha,\beta,c_0,\beta_0)>0$ such that for any $B_\rho(y) \subset B_r(y) \subset\subset \Omega$:
\begin{align}
\left( \int_{B_\rho(y)}{|\nabla u|^{\gamma_0}dx}\right)^{\frac{1}{\gamma_0}}\leq  C \left(\mathbf{M}_{\theta}^{T_0}(|\mu|)(y) \right)^{\frac{1}{p-1}}\rho^{\frac{n}{\gamma_0}+\delta-1},
\end{align}
where $\theta=1+(p-1)(1-\delta)$, $T_0 = diam(\Omega)$, and $0<\rho<T_0$.
\end{lemma}

In order to prove this Lemma \ref{lem:res7.2H}, it will be necessary to refer to Lemma \ref{lem:hanlin} in \cite[Lemma 1.4]{Han-Lin} as follows, where its proof can be found therein.
\begin{lemma}
\label{lem:hanlin}
Let $\phi(t)$ be a nonnegative and nondecreasing function on $[0,R]$. Suppose that 
$$
\phi(\rho) \le A \left[\left(\frac{\rho}{r} \right)^\alpha + \varepsilon \right] \phi(r) + B r^\beta,
$$
for any $0<\rho \le \theta r <R$, with $A, B, \alpha, \beta$ nonnegative constants and $\theta\in (0,1)$ and $\beta<\alpha$. Then, for any $\gamma \in (\beta,\alpha)$, there exists a constant $\varepsilon_0=\varepsilon_0(A,\alpha,\beta,\gamma,\theta)$ such that if $\varepsilon<\varepsilon_0$ we have for all $0<\rho \le r \le R$:
$$
\phi(\rho) \le C \left[\left(\frac{\rho}{r} \right)^\gamma \phi(r) + B\rho^\beta \right],
$$
where $C$ is a positive constant depending on $A,\alpha,\beta,\gamma$. In particular, we have for any $0<r \le R$:
$$
\phi(r) \le C \left[\frac{\phi(R)}{R^\gamma} r^\gamma + Br^\beta \right].
$$
\end{lemma}
\begin{proof}[Proof of Lemma \ref{lem:res7.2H}]

First of all, for $0<\rho\le r/2$, let us take $B_{r}(y) \subset\subset \Omega$, where $B_\rho(y) \subset B_{r}(y)$. By making use of Lemma \ref{lem:estimateinter} with $B_{2R} = B_{r}(y)$, one gives:
\begin{align}\label{eq:bylem1}
\left(\fint_{B_{r}(y)}{|\nabla (u-w)|^{\gamma_0}dx} \right)^{\frac{1}{\gamma_0}}&\leq C\left[ \frac{|\mu|(B_{r}(y))}{r^{n-1}} \right]^{\frac{1}{p-1}}\\
	&\qquad +	C \frac{|\mu|(B_{r}(y))}{r^{n-1}} \left(	\fint_{B_{r}(y)}|\nabla u|^{\gamma_0}dx\right)^{\frac{2-p}{\gamma_0}},
\end{align}
and applying the Lemma \ref{lem:res3.3P} with $B_\rho(y) \subset B_{r}(y) \subset B_{2R}(x_0)$ and $p=\gamma_0$ shows that:
\begin{align}\label{eq:bylem2}
\left(\fint_{B_\rho(y)}{|\nabla w|^{\gamma_0}} \right)^{\frac{1}{\gamma_0}} \le C \left(\frac{\rho}{r} \right)^{\beta_0-1}\left(\fint_{B_{2r/3}(y)}{|\nabla w|^{\gamma_0}} \right)^{\frac{1}{\gamma_0}}.
\end{align}
Combining \eqref{eq:bylem1}, \eqref{eq:bylem2} with the fact that
\begin{align*}
 \int_{B_{2r/3}(y)}{|\nabla w|^{\gamma_0}}dx \le C \int_{B_{r}(y)}{|\nabla u|^{\gamma_0}}dx.
\end{align*}
we obtain
\begin{align*}
\left(\fint_{B_\rho(y)}{|\nabla u|^{\gamma_0}dx} \right)^{\frac{1}{\gamma_0}} &\leq \left(\fint_{B_\rho(y)}{|\nabla w|^{\gamma_0} dx} \right)^{\frac{1}{\gamma_0}} + \left(\fint_{B_\rho(y)}{|\nabla u-\nabla w|^{\gamma_0}dx} \right)^{\frac{1}{\gamma_0}}\\ &\leq C\left(\frac{\rho}{r} \right)^{\beta_0-1} \left(\fint_{B_{r(y)}}{|\nabla u|^{\gamma_0}dx} \right)^{\frac{1}{\gamma_0}}  + C\left(\frac{|\mu|(B_{r}(y))}{r^{n-1}} \right)^{\frac{1}{p-1}}\\ ~~~~~~ &+ C\frac{|\mu|(B_{r}(y))}{r^{n-1}} \left(\fint_{B_{r(y)}}{|\nabla u|^{\gamma_0}}dx \right)^{\frac{2-p}{\gamma_0}},
\end{align*}
which implies
\begin{align*}
\begin{split}
\left( \int_{B_\rho(y)}{|\nabla u|^{\gamma_0}dx}\right)^{\frac{1}{\gamma_0}} &\leq C\left(\frac{\rho}{r} \right)^{\frac{n}{\gamma_0}+\beta_0-1} \left(\int_{B_{r(y)}}{|\nabla u|^{\gamma_0}dx} \right)^{\frac{1}{\gamma_0}} \\ ~~~~~&+ C \rho^{\frac{n}{\gamma_0}}\left(\frac{|\mu|(B_{r}(y))}{r^{n-1}} \right)^{\frac{1}{p-1}}\\&~~+ C\rho^{\frac{n(p-1)}{\gamma_0}}\frac{|\mu|(B_{r}(y))}{r^{n-1}} \left(\frac{\rho}{r}\right)^{\frac{n(2-p)}{\gamma_0}}\left(\int_{B_{r(y)}}{|\nabla u|^{\gamma_0}}dx \right)^{\frac{2-p}{\gamma_0}}.
\end{split}
\end{align*}
Using H\"older's inequality for the last term, one finds 
\begin{align}
\label{eq:btholder}
\begin{split}
\left( \int_{B_\rho(y)}{|\nabla u|^{\gamma_0}dx}\right)^{\frac{1}{\gamma_0}} &\leq C\left(\frac{\rho}{r} \right)^{\frac{n}{\gamma_0}+\beta_0-1} \left(\int_{B_{r(y)}}{|\nabla u|^{\gamma_0}dx} \right)^{\frac{1}{\gamma_0}} \\ ~~~~~~~&+ C_\varepsilon \rho^{\frac{n}{\gamma_0}}\left(\frac{|\mu|(B_{r}(y))}{r^{n-1}} \right)^{\frac{1}{p-1}}\\&~~+ \varepsilon \left(\frac{\rho}{r}\right)^{\frac{n}{\gamma_0}}\left(\int_{B_{r(y)}}{|\nabla u|^{\gamma_0}}dx \right)^{\frac{1}{\gamma_0}}.
\end{split}
\end{align}
The repeated application of Lemma \ref{lem:hanlin} enables us to set function $\Phi: \mathbb{R} \to \mathbb{R}$ of any $t \in \mathbb{R}, t>0$ defined as:
\begin{align}\label{eq:Phi}
\Phi(t)=\left( \int_{B_t(y)}{|\nabla u|^{\gamma_0}dx}\right)^{\frac{1}{\gamma_0}}.
\end{align}
Thus, \eqref{eq:btholder} can be rewritten in term of function $\Phi$:
\begin{align}
\begin{split}
\Phi(\rho)\leq C\left[\left(\frac{\rho}{r} \right)^{\frac{n}{\gamma_0}+\beta_0-1} +\varepsilon\right]\Phi(r) + C_\varepsilon \rho^{\frac{n}{\gamma_0}}\left(\frac{|\mu|(B_{r}(y))}{r^{n-1}} \right)^{\frac{1}{p-1}}.
\end{split}
\end{align}
Therefore, for any $\delta \in \left[-\frac{n-p}{p-1},\beta_0 \right)$, it satisfies that
\begin{align*}
\begin{split}
\Phi(\rho)&\leq C\left[\left(\frac{\rho}{r} \right)^{\frac{n}{\gamma_0}+\beta_0-1} +\varepsilon\right]\Phi(r) + C_\varepsilon r^{\frac{n}{\gamma_0}+\delta-1}r^{1-\delta}\left(\frac{|\mu|(B_{r}(y))}{r^{n-1}} \right)^{\frac{1}{p-1}}\\& = C\left[\left(\frac{\rho}{r} \right)^{\frac{n}{\gamma_0}+\beta_0-1} +\varepsilon\right]\Phi(r) + C_\varepsilon r^{\frac{n}{\gamma_0}+\delta-1}\left(\frac{|\mu|(B_{r}(y))}{r^{n-1-(p-1)(1-\delta)}} \right)^{\frac{1}{p-1}}.
\end{split}
\end{align*}
Choosing  $\varepsilon>0$ small enough, then applying Lemma \ref{lem:hanlin}, with $\gamma= \frac{n}{\gamma_0}+\beta_0-1,\beta = \frac{n}{\gamma_0}+\delta-1$, with $\delta<\beta_0$ (as $\beta<\gamma$ in Lemma \ref{lem:hanlin}) and $B= \left(\mathbf{M}_{\theta}^{T_0}(|\mu|)(y) \right)^{\frac{1}{p-1}}$, for any $0<\rho<r<T_0$ it is easily seen that
\begin{align*}
\Phi(\rho) \le C \left[\left(\frac{\rho}{r} \right)^{\frac{n}{\gamma_0}+\beta_0-1} \Phi(r) + C\rho^{\frac{n}{\gamma_0}+\delta-1}  \left(\mathbf{M}_{\theta}^{T_0}(|\mu|)(y) \right)^{\frac{1}{p-1}}\right],
\end{align*}
and we thus get
\begin{align}
\label{eq:inMtheta}
\begin{split}
&\left( \int_{B_\rho(y)}{|\nabla u|^{\gamma_0}dx}\right)^{\frac{1}{\gamma_0}}\\& \leq  C \left[\left(\frac{1}{T_0} \right)^{\frac{n}{\gamma_0}+\delta-1} \left( \int_{\Omega}{|\nabla u|^{\gamma_0}dx}\right)^{\gamma_0} + \left(\mathbf{M}_{\theta}^{T_0}(|\mu|)(y) \right)^{\frac{1}{p-1}}\right] \rho^{\frac{n}{\gamma_0}+\delta-1}.
\end{split}
\end{align}
According to the Remark \ref{rem:nablau}, it gives
\begin{align*}
\left(	\frac{1}{T_0^n}\int_{\Omega}|\nabla u|^{\gamma}\right)^{1/\gamma}\leq C_\gamma \left[\frac{|\mu|(\Omega)}{T_0^{n-1}}\right]^{\frac{1}{p-1}}, \qquad \text{for any}\ \ \gamma\in \left(0,\frac{(p-1)n}{n-1}\right)
\end{align*}
which implies that
\begin{align}
\label{eq:2inM}
\begin{split}
\left(\frac{1}{T_0} \right)^{\frac{n}{\gamma_0}+\delta-1} \left( \int_{\Omega}{|\nabla u|^{\gamma_0}dx}\right)^{\frac{1}{\gamma_0}} & \leq C_\gamma \left(\frac{1}{T_0} \right)^{\delta-1} \left[\frac{|\mu|(\Omega)}{T_0^{n-1}}\right]^{\frac{1}{p-1}} \\
 & \leq C \left(\mathbf{M}_{\theta}^{T_0}(|\mu|)(y) \right)^{\frac{1}{p-1}}.
 \end{split}
\end{align}
From \eqref{eq:inMtheta} and \eqref{eq:2inM} it turns to
\begin{align*}
\left( \int_{B_\rho(y)}{|\nabla u|^{\gamma_0}dx}\right)^{\frac{1}{\gamma_0}}\leq  C \left(\mathbf{M}_{\theta}^{T_0}(|\mu|)(y) \right)^{\frac{1}{p-1}}\rho^{\frac{n}{\gamma_0}+\delta-1},
\end{align*}
and that is our desired conclusion.
\end{proof}

\subsection{Boundary Estimates}
Next, let us give some comparison estimates on the boundary, the same conclusion as interior estimates can be drawn. First, as $\mathbb{R}^n \setminus \Omega$ is uniformly $p$-thick with constants $c_0, r_0>0$, let $x_0 \in \partial \Omega$ be a boundary point and for $0<R<r_0/10$ we set $\Omega_{10R} = \Omega_{10R}(x_0) = B_{10R}(x_0) \cap \Omega$. With $u \in W^{1,p}_0(\Omega)$ being a solution to \eqref{eq:elliptictype}, we consider the unique solution $w \in u+W^{1,p}_0(\Omega_{10R})$ to the following equation:
\begin{equation}
\label{111120146*}
\left\{ \begin{array}{rcl}
- \operatorname{div}\left( {A(x,\nabla w)} \right) &=& 0 \quad ~~~\text{in}\quad \Omega_{10R}(x_0), \\ 
w &=& u\quad \quad \text{on} \quad \partial \Omega_{10R}(x_0). 
\end{array} \right.
\end{equation}
In what follows we extend $\mu$ and $u$ by zero to $\mathbb{R}^n \setminus \Omega$ and $w$ by $u$ to $\mathbb{R}^n \setminus \Omega_{10R}$. Let us recall the following Lemma \ref{111120147*}, which was stated and proved in \cite{55Ph0}. This naturally leads to Lemma \ref{111120147**} below, whose proof can be found in \cite{MP2018}.
\begin{lemma} \label{111120147*} 
Let $w$ be the solution to \eqref{111120146*}. Then, there exist  constants $\Theta=\Theta(n,p,\alpha, \beta,c_0)>p$ and $C = C(n,p,\alpha,\beta,c_0)>0$ such that the following estimate    
	\begin{equation}\label{111120148*}
	\left(\fint_{B_{\rho/2}(y)}|\nabla w|^{\Theta} dx\right)^{\frac{1}{\Theta}}\leq C\left(\fint_{B_{3\rho}(y)}|\nabla w|^{p-1} dx\right)^{\frac{1}{p-1}}
	\end{equation} 
holds for all  $B_{3\rho}(y) \subset B_{10R}(x_0)$, $y \in B_r(x_0)$. 
\end{lemma}

\begin{lemma} 
\label{111120147**} 
Let $w$ be the solution to \eqref{111120146*}. Then, there exist  constants $\Theta=\Theta(n,p,\alpha, \beta,c_0)>p$ and $C = C(n,p,\alpha,\beta,c_0)>0$ such that we have the following estimate      
	\begin{equation}
	\label{111120148**}
	\left(\fint_{B_{\rho/2}(y)}|\nabla w|^{\Theta} dx\right)^{\frac{1}{\Theta}}\leq C\left(\fint_{B_{2\rho/3}(y)}|\nabla w|^{p-1} dx\right)^{\frac{1}{p-1}}
	\end{equation} 
	holds for all  $B_{\rho}(y)\subset B_{10R}(x_0)$, $y \in B_r(x_0)$. 
\end{lemma}
Now, we state the following Lemma in a somewhat more general form of Lemmas \ref{111120147*} and \ref{111120147**}.
 \begin{lemma}
\label{lem:addon12}
Let $w$ be the solution to \eqref{111120146*}. Then, for $0<\theta_1<\theta_2<1$ there exist  constants $\Theta=\Theta(n,p,\alpha, \beta,c_0)>p$ and $C = C(n,p,\alpha,\beta,\theta_1, \theta_2,c_0)>0$ such that we have the following estimate      
	\begin{equation}\label{111120148**}
	\left(\fint_{B_{\theta_1\rho}(y)}|\nabla w|^{\Theta} dx\right)^{\frac{1}{\Theta}}\leq C\left(\fint_{B_{\theta_2\rho}(y)}|\nabla w|^{p-1} dx\right)^{\frac{1}{p-1}}
	\end{equation} 
	holds for all  $B_\rho(y) \subset B_{10R}(x_0)$, $y \in B_r(x_0)$. 
\end{lemma}
More formally, Lemmas \ref{lem:estimatebound}, \ref{lem:res3.7P} and \ref{lem:res3.8P}, which we state below are main ingredients for us to obtain boundary estimates. They are the boundary version of Lemmas \ref{lem:estimateinter}, \ref{lem:res3.3P} and \ref{lem:res3.4P}, respectively.
\begin{lemma}
\label{lem:estimatebound}
Let $w$ be the solution to \eqref{111120146*}. Then, for any $\frac{2-p}{2}\leq \gamma_0<\frac{(p-1)n}{n-1}\leq 1$, there is a constant $C = C(n,p,\alpha,\beta,c_0)>0$ such that:
	\begin{align}
	\label{eq:estimateinter}
	\begin{split}
	\left(	\fint_{B_{10R}(x_0)}|\nabla (u-w)|^{\gamma_0}dx\right)^{\frac{1}{\gamma_0}}&\leq C\left[ 	\frac{|\mu|(B_{10R}(x_0))}{R^{n-1}}\right]^{\frac{1}{p-1}}\\
	&\qquad +	C 	\frac{|\mu|(B_{10R}(x_0))}{R^{n-1}} \left(	\fint_{B_{10R}(x_0)}|\nabla u|^{\gamma_0}dx\right)^{\frac{2-p}{\gamma_0}}.
	\end{split}
	\end{align}
	\end{lemma}
\begin{lemma}
\label{lem:res3.7P}
Let $w$ be the solution to \eqref{111120146*}. Then, there exist constants $\beta_0 = \beta_0(n,p,\alpha,\beta,c_0) \in (0,1/2]$ and $C =C(n,p,\alpha,\beta,c_0)>0$ such that:
\begin{align}
\left( \fint_{B_\rho(y)}{|\nabla w|^p dx} \right)^{\frac{1}{p}} \le C \left(\frac{\rho}{r} \right)^{\beta_0} \left(\fint_{B_r(y)}{|\nabla w|^p dx} \right)^{\frac{1}{p}},
\end{align}
for any $y \in B_{r}(x_0)$ such that $B_\rho(y) \subset B_r(y) \subset B_{10R}(x_0)$.
\end{lemma}
\begin{lemma}
\label{lem:res3.8P}
Let $w$ be the solution to \eqref{111120146*}. Then, for any $\Theta \in (0,p]$, there exist constants $\beta_0 = \beta_0(n,p,\alpha,\beta,c_0) \in (0,1/2]$ and $C =C(n,p,\alpha,\beta,\Theta,c_0)>0$ there holds:
\begin{align}
\left(\fint_{B_\rho(y)}{|\nabla w|^\Theta dx} \right)^{\frac{1}{\Theta}} \le C \left(\frac{\rho}{r} \right)^{\beta_0-1}\left( \fint_{B_r(y)}{|\nabla w|^\Theta dx} \right)^{\frac{1}{\Theta}},
\end{align}
for any $y \in B_{r}(x_0)$ such that $B_\rho(y) \subset B_r(y) \subset B_{10R}(x_0)$.
\end{lemma}
We next state and prove the selection Lemma which establishes solution gradient estimate up to the boundary.
\begin{lemma}
\label{lem:res7.6H}
Let $\beta_0 \in (0,1/2]$ be as in Lemmas \ref{lem:res3.7P} and \ref{lem:res3.8P}. Then, for any $\delta \in \left[-\frac{n-p}{p-1},\beta_0\right)$, there exists a constant $C = C(n,p,\alpha,\beta,c_0,\beta_0)>0$ such that for any $B_\rho(y) \cap \partial\Omega \neq \emptyset$:
\begin{align}
\left( \int_{B_\rho(y)}{|\nabla u|^{\gamma_0}dx}\right)^{\frac{1}{\gamma_0}}\leq  C \left(\mathbf{M}_{\theta}^{T_0}(|\mu|)(y) \right)^{\frac{1}{p-1}}\rho^{\frac{n}{\gamma_0}+\delta-1},
\end{align}
where $\theta = 1+(p-1)(1-\delta)$, $T_0 = diam(\Omega)$, $0<\rho<T_0$.
\end{lemma}
\begin{proof}[Proof of Lemma \ref{lem:res7.6H}]

We begin by taking $0<\rho' \le 2r$ such that $B_{\rho'/4}(y) \cap \partial\Omega \neq \emptyset$. Let $y_0 \in B_{\rho'/4}(y) \cap \partial\Omega$ such that $|y-y_0| = dist(y,\partial\Omega) \le \rho'/4$. Note that since $|y-y_0| \le \rho'/4$, this clearly gives $B_{\rho'/4}(y_0) \subset B_{\rho'/2}(y)$ and $B_{\rho'/2}(y_0) \subset B_{3\rho'/4}(y)$. Otherwise, for $\rho \ge |y-y_0|/4$ claims that $B_\rho(y) \subset B_{5\rho}(y_0)$.

For $\rho \le \rho'/4$, let $w$ be as in Lemmas \ref{lem:res3.7P} and \ref{lem:res3.8P} with $B_\rho(y) \subset B_{\rho'}(y) \subset B_{10R}(x_0)$. Applying $B_{10R} = B_{\rho'/2}(y_0)$ in Lemma \ref{lem:estimatebound} yields:
\begin{align}\label{eq:st1}
	\begin{split}
	\left(	\fint_{B_{\rho'}(y_0)}|\nabla (u-w)|^{\gamma_0}dx\right)^{\frac{1}{\gamma_0}}&\leq C\left[ 	\frac{|\mu|(B_{\rho'}(y_0))}{{\rho'}^{n-1}}\right]^{\frac{1}{p-1}}\\
	&\qquad +	C 	\frac{|\mu|(B_{\rho'}(y_0))}{{\rho'}^{n-1}} \left(	\fint_{B_{\rho'}(y_0)}|\nabla u|^{\gamma_0}dx\right)^{\frac{2-p}{\gamma_0}}.
	\end{split}
\end{align}
Applying Lemma \ref{lem:res3.7P} with $B_{\rho}(y) \subset B_{\rho'}(y) \subset B_{10R}$ and $p=\gamma_0$ shows that:
\begin{align*}
\left( \fint_{B_\rho(y)}{|\nabla w|^{\gamma_0} dx} \right)^{\frac{1}{\gamma_0}} \le C \left(\frac{\rho}{\rho'} \right)^{\beta_0} \left(\fint_{B_3\rho'/4(y)}{|\nabla w|^{\gamma_0} dx} \right)^{\frac{1}{\gamma_0}}.
\end{align*}
Moreover, since:
\begin{align*}
\int_{B_{\rho}(y)}{|\nabla w|^{\gamma_0}dx} \le C \int_{B_{5\rho(y_0)}}{|\nabla u|^{\gamma_0}dx},
\end{align*}
it follows that
\begin{align}
\label{eq:st2}
\begin{split}
\left(\fint_{B_\rho(y)}{|\nabla w|^{\gamma_0}dx} \right)^{\frac{1}{\gamma_0}} &\le C \left(\fint_{B_{5\rho}(y_0)}{|\nabla u|^{\gamma_0}dx} \right)^{\frac{1}{\gamma_0}}\\  &\le C \left(\frac{\rho}{\rho'} \right)^{\beta_0-1}\left(\fint_{B_{\rho'/4}(y_0)}{|\nabla u|^{\gamma_0}dx} \right)^{\frac{1}{\gamma_0}}\\ &\le C \left(\frac{\rho}{\rho'} \right)^{\beta_0-1}\left(\fint_{B_{\rho'/2}(y)}{|\nabla u|^{\gamma_0}dx} \right)^{\frac{1}{\gamma_0}}.
\end{split}
\end{align}
According to \eqref{eq:st1} and \eqref{eq:st2}, we get the estimate:
\begin{align*}
\left(\fint_{B_\rho(y)}{|\nabla u|^{\gamma_0}dx} \right)^{\frac{1}{\gamma_0}} &\leq \left(\fint_{B_\rho(y)}{|\nabla w|^{\gamma_0} dx} \right)^{\frac{1}{\gamma_0}} + \left(\fint_{B_\rho(y)}{|\nabla u-\nabla w|^{\gamma_0}dx} \right)^{\frac{1}{\gamma_0}}\\ &\leq C\left(\frac{\rho}{\rho'} \right)^{\beta_0-1} \left(\fint_{B_{\rho'/2(y)}}{|\nabla u|^{\gamma_0}dx} \right)^{\frac{1}{\gamma_0}}  + C\left(\frac{|\mu|(B_{\rho'}(y))}{\rho'^{n-1}} \right)^{\frac{1}{p-1}}\\ ~~~~~~ &+ C\frac{|\mu|(B_{\rho'}(y))}{\rho'^{n-1}} \left(\fint_{B_{\rho'(y)}}{|\nabla u|^{\gamma_0}}dx \right)^{\frac{2-p}{\gamma_0}}.
\end{align*}
Back to the integral on the corresponding ball, it gives
\begin{align}
\label{eq:st3}
\begin{split}
\left( \int_{B_\rho(y)}{|\nabla u|^{\gamma_0}dx}\right)^{\frac{1}{\gamma_0}} &\leq C\left(\frac{\rho}{\rho'} \right)^{\frac{n}{\gamma_0}+\beta_0-1} \left(\int_{B_{\rho'(y)}}{|\nabla u|^{\gamma_0}dx} \right)^{\frac{1}{\gamma_0}} \\ ~~~~~&+ C \rho^{\frac{n}{\gamma_0}}\left(\frac{|\mu|(B_{\rho'}(y))}{\rho'^{n-1}} \right)^{\frac{1}{p-1}}\\&~~+ C\rho^{\frac{n(p-1)}{\gamma_0}}\frac{|\mu|(B_{\rho'}(y))}{\rho'^{n-1}} \left(\frac{\rho}{\rho'}\right)^{\frac{n(2-p)}{\gamma_0}}\left(\int_{B_{\rho'(y)}}{|\nabla u|^{\gamma_0}}dx \right)^{\frac{2-p}{\gamma_0}}.
\end{split}
\end{align}
In the use of H\"older's inequality for the last term, it may be concluded that:
\begin{align}
\label{eq:btholderbound}
\begin{split}
\left( \int_{B_\rho(y)}{|\nabla u|^{\gamma_0}dx}\right)^{\frac{1}{\gamma_0}} &\leq C\left(\frac{\rho}{\rho'} \right)^{\frac{n}{\gamma_0}+\beta_0-1} \left(\int_{B_{\rho'(y)}}{|\nabla u|^{\gamma_0}dx} \right)^{\frac{1}{\gamma_0}} \\ ~~~~~&+ C_\varepsilon \rho^{\frac{n}{\gamma_0}}\left(\frac{|\mu|(B_{\rho'}(y))}{\rho'^{n-1}} \right)^{\frac{1}{p-1}}\\&~~+ \varepsilon \left(\frac{\rho}{\rho'}\right)^{\frac{n}{\gamma_0}}\left(\int_{B_{\rho'(y)}}{|\nabla u|^{\gamma_0}}dx \right)^{\frac{1}{\gamma_0}}.
\end{split}
\end{align}
The same proof works when we take Lemma \ref{lem:hanlin} in use. It is natural to set a function $\Phi$ as in \eqref{eq:Phi}. From \eqref{eq:btholderbound}, it can be written the inequality of the form
\begin{align*}
\begin{split}
\Phi(\rho)\leq C\left[\left(\frac{\rho}{\rho'} \right)^{\frac{n}{\gamma_0}+\beta_0-1} +\varepsilon\right]\Phi(\rho') + C_\varepsilon \rho^{\frac{n}{\gamma_0}}\left(\frac{|\mu|(B_{\rho'}(y))}{\rho'^{n-1}} \right)^{\frac{1}{p-1}}.
\end{split}
\end{align*}
It suffices to show that for any $\delta \in \left[-\frac{n-p}{p-1},\beta_0 \right)$, we have
\begin{align*}
\begin{split}
\Phi(\rho)&\leq C\left[\left(\frac{\rho}{\rho'} \right)^{\frac{n}{\gamma_0}+\beta_0-1} +\varepsilon\right]\Phi(\rho') + C_\varepsilon \rho'^{\frac{n}{\gamma_0}+\delta-1}\rho'^{1-\delta}\left(\frac{|\mu|(B_{\rho'}(y))}{\rho'^{n-1}} \right)^{\frac{1}{p-1}}\\& = C\left[\left(\frac{\rho}{\rho'} \right)^{\frac{n}{\gamma_0}+\beta_0-1} +\varepsilon\right]\Phi(\rho') + C_\varepsilon \rho'^{\frac{n}{\gamma_0}+\delta-1}\left(\frac{|\mu|(B_{\rho'}(y))}{\rho'^{n-1-(p-1)(1-\delta)}} \right)^{\frac{1}{p-1}},
\end{split}
\end{align*}
and here the supremum is taken for all $0<\rho'<T_0$, that yields
\begin{align*}
\begin{split}
\Phi(\rho)&\leq C\left[\left(\frac{\rho}{\rho'} \right)^{\frac{n}{\gamma_0}+\beta_0-1} +\varepsilon\right]\Phi(\rho') + C_\varepsilon \rho'^{\frac{n}{\gamma_0}+\delta-1}\left(\mathbf{M}_{\theta}^{T_0}(|\mu|)(y) \right)^{\frac{1}{p-1}}.
\end{split}
\end{align*}
For  $\varepsilon>0$ is chosen small enough, we apply Lemma \ref{lem:hanlin} where $\gamma= \frac{n}{\gamma_0}+\beta_0-1,\beta = \frac{n}{\gamma_0}+\delta-1$ with $\delta<\beta_0$ (as $\beta<\gamma$ in Lemma \ref{lem:hanlin}) and $B= \left(\mathbf{M}_{\theta}^{T_0}(|\mu|)(y) \right)^{\frac{1}{p-1}}$, for any $0<\rho<\rho'<T_0$ one gets
\begin{align*}
\Phi(\rho) \le C \left[\left(\frac{\rho}{\rho'} \right)^{\frac{n}{\gamma_0}+\beta_0-1} \Phi(\rho') + C\rho^{\frac{n}{\gamma_0}+\delta-1}  \left(\mathbf{M}_{\theta}^{T_0}(|\mu|)(y) \right)^{\frac{1}{p-1}}\right].
\end{align*}
Hence,
\begin{align}
\label{eq:inMthetabound}
\begin{split}
&\left( \int_{B_\rho(y)}{|\nabla u|^{\gamma_0}dx}\right)^{\frac{1}{\gamma_0}}\\& \leq  C \left[\left(\frac{1}{T_0} \right)^{\frac{n}{\gamma_0}+\delta-1} \left( \int_{\Omega}{|\nabla u|^{\gamma_0}dx}\right)^{\gamma_0} + \left(\mathbf{M}_{\theta}^{T_0}(|\mu|)(y) \right)^{\frac{1}{p-1}}\right] \rho^{\frac{n}{\gamma_0}+\delta-1}.
\end{split}
\end{align}
It follows easily from Remark \ref{rem:nablau} that 
\begin{align*}
\left(	\frac{1}{T_0^n}\int_{\Omega}|\nabla u|^{\gamma}\right)^{1/\gamma}\leq C_\gamma \left[\frac{|\mu|(\Omega)}{T_0^{n-1}}\right]^{\frac{1}{p-1}}, \quad \text{for any}\ \ \gamma\in \left(0,\frac{(p-1)n}{n-1}\right)
\end{align*}
which implies
\begin{align}
\label{eq:2inMbound}
\left(\frac{1}{T_0} \right)^{\frac{n}{\gamma_0}+\delta-1} \left( \int_{\Omega}{|\nabla u|^{\gamma_0}dx}\right)^{\frac{1}{\gamma_0}} \leq C \left(\frac{1}{T_0} \right)^{\delta-1} \left[\frac{|\mu|(\Omega)}{T_0^{n-1}}\right]^{\frac{1}{p-1}}\leq C \left(\mathbf{M}_{\theta}^{T_0}(|\mu|)(y) \right)^{\frac{1}{p-1}}.
\end{align}
From both \eqref{eq:inMthetabound} and \eqref{eq:2inMbound} have already proved, it completes the proof.
\end{proof}

As a consequence of Lemmas \ref{lem:res7.2H} and \ref{lem:res7.6H}, the supremum is taken for all $0<\rho<T_0$ and $y \in \Omega$, one of our results may be summarized in the following important Lemma.
\begin{lemma}
\label{lem:H:Sep26}
Let $\beta_0 \in (0,1/2]$ be as in Lemmas \ref{lem:res3.4P} and \ref{lem:res3.8P}. Then, for any $\delta \in \left[ -\frac{n-p}{p-1},\beta_0\right)$, there exists a constant $C = C(n,p,\alpha,\beta,c_0,\beta_0)>0$ such that:
\begin{align}
\label{eq:7276}
\sup_{y\in \Omega, \rho\in (0,T_0)} \rho^{-\frac{n}{\gamma_0}-\delta+1} \left( \int_{B_\rho(y)}{|\nabla u|^{\gamma_0}dx}\right)^{\frac{1}{\gamma_0}}\leq C\|\mathbf{M}_{\theta}^{T_0}(|\mu|)\|_{L^\infty(\Omega)}^{\frac{1}{p-1}},
\end{align}
where $\theta = 1+(p-1)(1-\delta)$, $T_0 = diam(\Omega)$, $0<\rho<T_0$.
\end{lemma}

\section{Proofs of Theorem \ref{theo:lambda_estimate} and \ref{theo:P}}
\label{sec:main}
This section is devoted to separable proofs of our main results in Theorem \ref{theo:lambda_estimate}, \ref{theo:bst} and \ref{theo:P}. Here, our proof techniques are global up to the boundary results.

We first prove the form of Theorem \ref{theo:lambda_estimate}. The main tools are properties of Hardy-Littlewood maximal function and the following lemma, that can be viewed as a substitution for the Calder\'on-Zygmund-Krylov-Safonov decomposition.

\begin{lemma}
\label{lem:mainlem}
Let $0<\varepsilon<1, R>0$ and the ball $Q:=B_R(x_0)$ for some $x_0\in \mathbb{R}^n$.  Let $E\subset F\subset Q$ be two measurable sets in $\mathbb{R}^{n+1}$ with $\mathcal{L}^n(E)<\varepsilon \mathcal{L}^n(Q)$ and satisfying the following property: for all $x\in Q$ and $r\in (0,R]$, we have $B_r(x)\cap Q\subset F$  	provided $\mathcal{L}^n(E\cap B_r(x))\geq \varepsilon \mathcal{L}^n(B_r(x))$.	Then $\mathcal{L}^n(E)\leq C\varepsilon \mathcal{L}^n(F)$ for some $C=C(n)$.
\end{lemma}

\begin{proof}[Proof of Theorem \ref{theo:lambda_estimate}]
Let $\mu_0,\lambda_k^+,\lambda_k^-$ be as in Definition \ref{def:renormsol3}. Let $u$ be the renormalized solution to \eqref{eq:elliptictype} and $u_k\in W_0^{1,p}(\Omega)$ be the unique solution to the following problem:
	\begin{equation*}
	\left\{
	\begin{array}[c]{rcl}
	-\text{div}(A(x,\nabla u_k))&=&\mu_{k} \quad \text{in } \quad\Omega,\\
	{u}_{k}&=&0\quad \ \text{on } \quad \partial\Omega,\\
	\end{array}
	\right.  
	\end{equation*}
	where  $\mu_k=\chi_{\{|u|<k\}}\mu_0+\lambda_k^+-\lambda_k^-$.
	
Note that we have  $u_k=T_k(u) \in W^{1,p}(\Omega)$ and $\mu_k \rightarrow \mu$ weakly in the narrow topology of measures (\ref{def:renormsol3}). In this way one obtains $\nabla u_k - \nabla u = \nabla u \cdot \chi_{\{|u|\ge k\}}$ almost everywhere and it is easily seen that
	\begin{equation}
	\label{grad-appr}
	\nabla u_k\to \nabla u  \quad \text{in} \quad L^\gamma(\Omega), \quad\forall \gamma\in \left(0,\frac{(p-1)n}{n-1}\right).
	\end{equation}

Let $\gamma_0$ satisfy $\frac{2-p}{2} \le \gamma_0 < \frac{(p-1)n}{n-1} \le 1$ given as in Lemmas \ref{lem:estimateinter} and \ref{lem:estimatebound}. For given $\varepsilon>0, \lambda>0$ and $r_0>0$ as in \eqref{eq:capuni}, let us set 
\begin{align*}
E_{\lambda,\varepsilon}=\left\{({\bf M}(\chi_{D_2}|\nabla u|^{\gamma_0}))^{1/\gamma_0}>\varepsilon^{-\frac{1}{\Theta}}\lambda, (\mathbf{M}_1(\chi_{D_2}\mu))^{\frac{1}{p-1}}\le \varepsilon^{\frac{1}{(p-1)\gamma_0}}\lambda \right\}\cap D_1,
\end{align*}
and 
\begin{align*}
F_\lambda=\left\{ ({\bf M}(\chi_{D_2}|\nabla u|^{\gamma_0}))^{1/\gamma_0}> \lambda\right\}\cap D_1,
\end{align*}
for $\lambda>\varepsilon^{-\frac{1}{(p-1)\gamma_0}}||\nabla u||_{L^{\gamma_0}(D_2)}R^{-\frac{n}{\gamma_0}}$. 

Our purpose here is to prove that there exist $\Theta$ and $C$ such that \eqref{eq:mainlambda} holds, it means that $\mathcal{L}^n\left( E_{\lambda,\varepsilon}\right) \le C\varepsilon \mathcal{L}^n \left( F_\lambda\right)$. And our proof proceeds in a series of two simple steps to verify Lemma \ref{lem:mainlem}. 

\emph{First step}. Let us verify that $$\mathcal{L}^n(E_{\lambda,\varepsilon}) < \varepsilon\mathcal{L}^n(B_R),$$
	  for all $\lambda>\varepsilon^{-\frac{2}{\gamma_0}+\frac{1}{\Theta}}\left\|\nabla u\right\|_{L^{\gamma_0}(D_2)}R^{-\frac{n}{\gamma_0}}$.
	  Indeed, without loss of generality, we may assume that $E_{\lambda,\varepsilon}\not=\emptyset$ (cause if $E_{\lambda,\varepsilon}=\emptyset$, it holds obiviously). Then, for the range of singular $p$, it gives
	\begin{align}
	\label{eq:bt3}
	|\mu| (\Omega)\leq T_0^{n-1}\left(\varepsilon^{\frac{1}{(p-1)\gamma_0}}\lambda\right)^{p-1}.
	\end{align}
	Since ${\bf M}$ is a bounded operator from $L^1(\mathbb{R}^{n})$ into $L^{1,\infty}(\mathbb{R}^{n})$, following Remark \ref{rem:boundM} with $s=1$, $g = |\nabla u|^{\gamma_0}$ and $t=\left(\varepsilon^{-\frac{1}{\Theta}}\lambda\right)^{\gamma_0}$, for all $\lambda>\varepsilon^{-\frac{1}{(p-1)\gamma_0}}||\nabla u||_{L^{\gamma_0}(D_2)}R^{-\frac{n}{\gamma_0}}$ we will get
	 \begin{align}
	 \label{eq:bt5}
	 \begin{split}
	  \mathcal{L}^n\left( E_{\lambda,\varepsilon}\right)&\leq \frac{C}{\left(\varepsilon^{-\frac{1}{\Theta}}\lambda\right)^{\gamma_0}}\int_{D_2}|\nabla u|^{\gamma_0}dx\\
	  &\leq  \frac{C}{\left(\varepsilon^{-\frac{1}{\Theta}}\varepsilon^{-\frac{2}{\gamma_0}+\frac{1}{\Theta}}\left\|\nabla u\right\|_{L^{\gamma_0}(D_2)}R^{-\frac{n}{\gamma_0}}\right)^{\gamma_0}}\int_{D_2}|\nabla u|^{\gamma_0}dx\\
	  &\leq C\varepsilon^{2}\mathcal{L}^n\left(D_2\right).
	  \end{split}
	 \end{align}
	  It is clear to apply Remark \ref{rem:nablau}, with $\gamma=\gamma_0$ shows that:
\begin{align}
\label{es14}
\left(	\frac{1}{T_0^n}\int_{\Omega}|\nabla u|^{\gamma_0}\right)^{1/\gamma_0}\leq C_{\gamma_0} \left[\frac{|\mu|(\Omega)}{T_0^{n-1}}\right]^{\frac{1}{p-1}}.
\end{align}
Then, in the use of \eqref{es14} and \eqref{eq:bt5} we get that
\begin{align}
\begin{split}
\mathcal{L}^n\left( E_{\lambda,\varepsilon}\right) &\leq \frac{CT_0^n}{\left(\varepsilon^{-\frac{1}{\Theta}}\lambda\right)^{\gamma_0}} \left[\frac{T_0^{n-1}\left(\varepsilon^{\frac{1}{(p-1)\gamma_0}}\lambda\right)^{p-1}}{T_0^{n-1}}\right]^{\frac{\gamma_0}{p-1}} \\
&=C\varepsilon^{\left(\frac{1}{\Theta}+\frac{1}{(p-1)\gamma_0}\right)\gamma_0}\mathcal{L}^n\left( B_{R}\right) \\ 
&< C\varepsilon\mathcal{L}^n\left( B_{R}\right).
\end{split}
\end{align}
\emph{Second step}. It is sufficient for us to verify that for all $x\in D_1$, $r\in (0,R]$, and $\lambda>\varepsilon^{-\frac{2}{\gamma_0}+\frac{1}{\Theta}}\left\|\nabla u\right\|_{L^{\gamma_0}(D_2)}R^{-\frac{n}{\gamma_0}}$ one has:
	 \begin{equation}\label{eq:bt6}
	\mathcal{L}^n\left( E_{\lambda,\varepsilon}\cap B_r(x)\right)\geq C\varepsilon \mathcal{L}^n\left(B_r(x)\right) \Longrightarrow B_r(x)\cap D_1\subset F_\lambda.
	 \end{equation}
	 	
	Indeed, let $x\in D_1$ and $0<r\leq R$, and by contradiction, let us assume that $B_r(x)\cap D_1\cap F^c_\lambda\not= \emptyset$ and $E_{\lambda,\varepsilon}\cap B_r(x)\not = \emptyset$. Then, there exist $x_1,x_2\in B_r(x)\cap D_1$ such that
	 \begin{align}
	 \label{eq:bt7}
	 \left[{\bf M}(\chi_{D_2}|\nabla u|^{\gamma_0})(x_1)\right]^{1/\gamma_0}\leq \lambda,
\end{align}	 	 
	 and 
	 \begin{align}
	 \label{eq:bt8}
	 \mathbf{M}_1(\chi_{D_2}\mu)(x_2)\le \left(\varepsilon^{ \frac{1}{(p-1)\gamma_0}} \lambda\right)^{p-1}.
	 \end{align}
	 One needs to prove that there exists a constant $C>0$ depending on $n,p,\alpha, \beta,c_0$ and $diam(\Omega)/{r_0}$ such that the following estimate holds:
	 \begin{equation}\label{5hh2310133}
	\mathcal{L}^n\left( E_{\lambda,\varepsilon}\cap B_r(x)\right)< C \varepsilon \mathcal{L}^n\left(B_r(x)\right). 
	 \end{equation}	 
	To do this, for $y \in B_r(x)$ and $\rho>0$, firstly we have
	\begin{align*}
	\left(\fint_{B_\rho(y)}\chi_{D_2}|\nabla u|^{\gamma_0}dx\right)^{1/\gamma_0}\leq  \sup\left\{\Lambda_1, \bar{\Lambda_1}\right\},
	\end{align*}
	where 
	\begin{align*}
	\begin{cases}
	\Lambda_1 &= \displaystyle{\left(\sup_{\rho'<r}\fint_{B_{\rho'}(y)}\chi_{B_{2r}(x)}\chi_{D_2}|\nabla u|^{\gamma_0}dx\right)^{1/\gamma_0}}, \\
	\bar{\Lambda_1} &= \displaystyle{\left(\sup_{\rho'\geq r}\fint_{B_{\rho'}(y)}\chi_{D_2}|\nabla u|^{\gamma_0}dx\right)^{1/\gamma_0}}.
	\end{cases}
	\end{align*}	
	For  $\rho'\geq r$, one has $B_{\rho'}(y)\subset B_{\rho'+r}(x)\subset B_{\rho'+2r}(x_1)\subset B_{3\rho'}(x_1)$. Thus, 
	\begin{align*}
	\left(\fint_{B_\rho(y)}\chi_{D_2}|\nabla u|^{\gamma_0}dx\right)^{1/\gamma_0}\leq  \sup\left\{\Lambda_2, \bar{\Lambda_2}\right\} \leq  \sup\left\{\Lambda_2, 3^{\frac{n}{\gamma_0}} \lambda\right\}.
	\end{align*}
	where 
	\begin{align*}
	\begin{cases}
	\Lambda_2 &= \displaystyle{\left[{\bf M}\left(\chi_{D_2}\chi_{B_{2r}(x)}|\nabla u|^{\gamma_0}\right)(y)\right]^{\frac{1}{\gamma_0}}},\\ 
	\bar{\Lambda_2} &= \displaystyle{3^{\frac{n}{\gamma_0}} \left(\sup_{\rho'>0}\fint_{B_{\rho'}(x_1)}\chi_{D_2}|\nabla u|^{\gamma_0}dx\right)^{1/\gamma_0}}.
	\end{cases}
	\end{align*}
	Taking the supremum both sides for $\rho>0$, it can be seen clearly that:
	 \begin{equation*}
	\left( {\bf M}(\chi_{D_2}|\nabla u|^{\gamma_0})(y)\right)^{1/\gamma_0}\leq \max\left\{\left[{\bf M}\left(\chi_{D_2}\chi_{B_{2r}(x)}|\nabla u|^{\gamma_0}\right)(y)\right]^{\frac{1}{\gamma_0}},3^{\frac{n}{\gamma_0}}\lambda\right\},~\forall y\in B_r(x).
	 \end{equation*}
	 Therefore, for all $\lambda>0$ and $\varepsilon$ satisfies $\varepsilon^{-\frac{1}{\Theta}}>3^{\frac{n}{\gamma_0}}$, it shows that
	 \begin{align}
	 \label{5hh2310134}
	 \begin{split}
	 E_{\lambda,\varepsilon}\cap B_r(x)=\left\{{\bf M}\left(\chi_{D_2}\chi_{B_{2r}(x)}|\nabla u|^{\gamma_0}\right)^{\frac{1}{\gamma_0}}>\varepsilon^{-\frac{1}{\Theta}}\lambda, (\mathbf{M}_{1}(\chi_{D_2}\mu))^{\frac{1}{p-1}}\leq\varepsilon^{\frac{1}{(p-1)\gamma_0}}\lambda\right\} \\ ~~~~~\cap D_1\cap B_r(x).
	 \end{split}
	 \end{align}
	 
	  In order to prove \eqref{5hh2310133} we separately consider for the case $B_{4r}(x)\subset\subset \Omega$ (interior) and the case $B_{4r}(x)\cap \Omega^{c}\not=\emptyset$ (boundary). The proof falls naturally into two cases. \\ \medskip\\
	   \textbf{Case 1: $B_{4r}(x)\subset\subset\Omega$}: Applying Lemma \ref{lem:estimateinter} for $u_k \in W^{1,p}_{0}(\Omega)$ and $w_k$ the solution to:
	\begin{equation}
	\label{eq:wsol}
	\begin{cases}
	\div (A(x,\nabla w_k)) &=0, \quad \text{in}\ \ B_{4r}(x),\\
	w_k &= u_k, \quad \text{on} \ \ \partial B_{4r}(x),
	\end{cases}
	\end{equation}
with $\mu = \mu_k$ and $B_{2R} = B_{4r}(x)$, one has a constant $C = C(n,p,\alpha,\beta,c_0,T_0/r_0)>0$ such that:
	\begin{align}
	\label{eq:btgeneral}
	\begin{split}
	\left(\fint_{B_{4r}(x)}{|\nabla u_k - \nabla w_k|^{\gamma_0}dx} \right)^{\frac{1}{\gamma_0}} &\le C\left[\frac{|\mu_k|(B_{4r}(x))}{r^{n-1}}\right]^{\frac{1}{p-1}}\\ &+ C\frac{|\mu_k|(B_{4r}(x))}{r^{n-1}} \left(\fint_{B_{4r}(x)}{|\nabla u_k|^{\gamma_0}dx} \right)^{\frac{2-p}{\gamma_0}}.
	\end{split}
	\end{align}
	Otherwise, Lemma \ref{111120147} is also applied to give:
	\begin{align}
\label{eq:bt17}
\begin{split}
	\left(\fint_{B_{2r}(x)}|\nabla w_k|^{\frac{1}{\Theta}} dx\right)^{\Theta} &\leq C \left(\fint_{B_{4r}(x)}|\nabla w_k|^{p-1} dx\right)^{\frac{1}{p-1}} \\ &\le C \left(\fint_{B_{4r}(x)}{|\nabla u_k|^{\gamma_0}dx} \right)^{\frac{1}{\gamma_0}}\\&+ C\left( \fint_{B_{4r}(x)}{|\nabla u_k - \nabla w_k|^{\gamma_0}dx} \right)^{\frac{1}{\gamma_0}},
\end{split}
\end{align}
where, the second inequality is obtained by using H\"older's inequality and for $\gamma_0>p-1$. On the other hand, it is easily to check that
	\begin{align}
	\label{eq:bt11}
	\begin{split}	
	& \mathcal{L}^n\left( E_{\lambda,\varepsilon}\cap B_r(x)\right) \\ &\leq \mathcal{L}^n\left( \{{\bf M}\left(\chi_{D_2}\chi_{B_{2r}(x)}|\nabla (u_k-w_k)|^{\gamma_0}\right)^{\frac{1}{\gamma_0}}>3^{-\frac{1}{\gamma_0}}\varepsilon^{- \frac{1}{\Theta}}\lambda\}\cap B_r(x)\right)\\&+ \mathcal{L}^n\left( \{{\bf M}\left(\chi_{D_2}\chi_{B_{2r}(x)}|\nabla (u-u_k)|^{\gamma_0}\right)^{\frac{1}{\gamma_0}}>3^{-\frac{1}{\gamma_0}}\varepsilon^{- \frac{1}{\Theta}}\lambda\}\cap B_r(x)\right) \\&+
	\mathcal{L}^n\left(  \{{\bf M}\left(\chi_{D_2}\chi_{B_{2r}(x)}|\nabla w_k|^{\gamma_0}\right)^{\frac{1}{\gamma_0}}>3^{-\frac{1}{\gamma_0}}\varepsilon^{- \frac{1}{\Theta}}\lambda\}\cap B_r(x)\right).  
	  \end{split}
	\end{align}
and by using Remark \ref{rem:boundM} for each term on right hand side of \eqref{eq:bt11}, one gives
	\begin{align}
	\label{eq:bt12}
	\begin{split}
	\mathcal{L}^n\left( E_{\lambda,\varepsilon}\cap B_r(x)\right) &\le \frac{C}{\left(\varepsilon^{-\frac{1}{\Theta}}\lambda\right)^{\gamma_0}}\left[ \int_{B_{2r}(x)}{\chi_{D_2}|\nabla u_k - \nabla w_k|^{\gamma_0}dx}+ \right. \\& \left. ~~~~~~~~~~~~~~~~~~~+ \int_{B_{2r}(x)}{\chi_{D_2}|\nabla u - \nabla u_k|^{\gamma_0}dx}\right]\\ &+\frac{C}{\left(\varepsilon^{-\frac{1}{\Theta}}\lambda\right)^{\Theta}}\int_{B_{2r}(x)}{\chi_{D_2}|\nabla w_k|^{\Theta}dx}.
	\end{split}
	\end{align}
	Combining inequalities \eqref{eq:btgeneral} with \eqref{eq:bt17} to \eqref{eq:bt12} yields
	\begin{align*}
	&\mathcal{L}^n\left( E_{\lambda,\varepsilon} \cap B_r(x)\right) \\&\le \varepsilon^{\gamma_0\frac{1}{\Theta}}\lambda^{-\gamma_0}r^n \left[C\left( \frac{|\mu_k|(B_{4r}(x))}{r^{n-1}} \right)^{\frac{1}{p-1}}+ \right.\\&\left.~~~~~~~~~~~~~~+C\frac{|\mu_k|(B_{4r}(x))}{r^{n-1}} \left(\fint_{B_{4r}(x)}{|\nabla u_k|^{\gamma_0}dx} \right)^{\frac{2-p}{\gamma_0}} \right]^{\gamma_0} \\ &+ C\varepsilon^{\gamma_0\frac{1}{\Theta}}\lambda^{-\gamma_0}\int_{B_{4r}(x)}{|\nabla u - \nabla u_k|^{\gamma_0}dx} \\
&+ C \varepsilon \lambda^{-\Theta}r^n\left[C\left(\int_{B_{4r}(x)}{|\nabla u_k|^{\gamma_0}dx} \right)^{\frac{1}{\gamma_0}} + C \left( \frac{|\mu_k|(B_{4r}(x))}{r^{n-1}} \right)^{\frac{1}{p-1}}+ \right.\\&\left.~~~~+ C\frac{|\mu_k|(B_{4r}(x))}{r^{n-1}} \left(\fint_{B_{4r}(x)}{|\nabla u_k|^{\gamma_0}dx} \right)^{\frac{2-p}{\gamma_0}} \right]^{\Theta}.
	\end{align*}
	Letting $k \to \infty$ and thanks to \eqref{eq:bt7} and \eqref{eq:bt8} one obtains:
	\begin{align*}
		&\mathcal{L}^n\left( E_{\lambda,\varepsilon} \cap B_r(x)\right) \\&\le \varepsilon^{\gamma_0\frac{1}{\Theta}}\lambda^{-\gamma_0} r^n \left[C\left( 	\frac{|\mu|(\overline{B_{4r}(x))}}{r^{n-1}} \right)^{\frac{1}{p-1}}+\right.\\&\left.~~~~~~~~~~~~~~+C	\frac{|\mu|(\overline{B_{4r}(x))}}{r^{n-1}} \left(\fint_{B_{4r}(x)}{|\nabla u|^{\gamma_0}dx} \right)^{\frac{2-p}{\gamma_0}} \right]^{\gamma_0} +
	\\	
		&+ C \varepsilon\lambda^{-\Theta}r^n\left[C\left(\int_{B_{4r}(x)}{|\nabla u|^{\gamma_0}dx} \right)^{\frac{1}{\gamma_0}} + C \left( \frac{|\mu|(\overline{B_{4r}(x))}}{r^{n-1}} \right)^{\frac{1}{p-1}}+ \right.\\&\left.~~~~+ C\frac{|\mu|(\overline{B_{4r}(x))}}{r^{n-1}} \left(\fint_{B_{4r}(x)}{|\nabla u|^{\gamma_0}dx} \right)^{\frac{2-p}{\gamma_0}} \right]^{\Theta}.
	\end{align*}
	As $|x-x_1|<r$, $B_{4r}(x)\subset B_{5r}(x_1)$. This gives:
		\begin{align}\label{eq:btp1}
		\begin{split}
	\fint_{B_{4r}(x)}|\nabla u|^{\gamma_0}dx&\leq  \frac{|B_5(0)|}{|B_4(0)|} 	\fint_{B_{5r}(x_1)}|\nabla u|^{\gamma_0}dx\\&\leq C\sup_{\rho>0} \fint_{B_{\rho}(x_1)}|\nabla u|^{\gamma_0}dx
	\\&= C\mathbf{M}\left(|\nabla u|^{\gamma_0}\right)(x_1).
	\end{split}
	\end{align}
	Similarly, as $|x-x_2|<r$, we obtain $B_{4r}(x)\subset B_{5r}(x_2)\subset D_2$ and for all $\rho>0$, it finds:
	\begin{align}\label{eq:btp2}
	\frac{|\mu|(\overline{B_{4r}(x)}))}{r^{n-1}} &\le \frac{|\mu|(B_{5\rho}(x_2))}{\rho^{n-1}}\leq 5^{n-1} \mathbf{M}_1(\chi_{D_2}\mu)(x_2).
	\end{align}	
	Applying \eqref{eq:btp1} and \eqref{eq:btp2} together with \eqref{eq:bt7}, \eqref{eq:bt8} yields that:	
	\begin{align*}
	\mathcal{L}^n\left(E_{\lambda,\varepsilon} \cap B_r(x)\right) &\le \varepsilon^{\gamma_0  \frac{1}{\Theta}+\gamma_0  \frac{1}{(p-1)\gamma_0}} r^n\left(C+C\varepsilon^{ \frac{1}{(p-1)\gamma_0} (p-2)}\right)^{\gamma_0}\\
	&+ C\varepsilon r^n\left(1+\varepsilon^{ \frac{1}{(p-1)\gamma_0}}+\varepsilon^{ \frac{1}{(p-1)\gamma_0}(p-1)}\right)^{\Theta}\\
	&\leq C\left[\varepsilon^{\gamma_0 \frac{1}{\Theta}+\gamma_0(p-1) \frac{1}{(p-1)\gamma_0}}+\varepsilon\right] r^n \\
	& \le C\varepsilon r^n.
	\end{align*} 
	which establishes our desired \eqref{5hh2310133}.\medskip\\
		 \textbf{Case 2: $B_{4r}(x) \cap \Omega^c \neq \emptyset$}:
	Let $x_3 \in \partial\Omega$ such that $|x_3-x|=\text{dist}(x,\partial\Omega)\leq 4r$. It is not difficult to check that:
	\begin{align*}
	B_{4r}(x) \subset B_{10r}(x_3)\subset D_2.
	\end{align*}
	Applying Lemma \ref{lem:estimatebound} for $u_k \in W^{1,p}_{0}(\Omega)$ and $w_k$ being solution to:
	\begin{equation}
	\label{eq:wsol10R}
	\begin{cases}
	\div(A(x,\nabla w_k)) &=0, \quad \text{in}\ \ B_{10r}(x_3),\\
	w_k &= u_k, \quad \text{on} \ \ \partial B_{10r}(x_3),
	\end{cases}
	\end{equation}
for $\mu = \mu_k$ and $B_{2R} = B_{10R}(x_3)$, one has a constant $C = C(n,p,\alpha,\beta, c_0,T_0/r_0)>0$ such that:
	\begin{align} 
	\label{eq:estbound1}
	\begin{split}
	\left(\fint_{B_{10r}(x_3)}{|\nabla u_k - \nabla w_k|^{\gamma_0}dx} \right)^{\frac{1}{\gamma_0}} &\le C\left[\frac{|\mu_k|(B_{10r}(x_3))}{r^{n-1}} \right]^{\frac{1}{p-1}}\\ &+ C\frac{|\mu_k|(B_{10r}(x_3))}{r^{n-1}} \left(\fint_{B_{10r}(x_3)}{|\nabla u_k|^{\gamma_0}dx} \right)^{\frac{2-p}{\gamma_0}},
	\end{split}
	\end{align}
	and for all $\rho>0$ satisfies $B_{\rho}(y)\subset B_{10r}(x_3)$, following Lemma \ref{111120147**} one has
	\begin{align}
	\label{eq:estbound2}
	\begin{split}
	\left(\fint_{B_{\rho/2}(y)}|\nabla w_k|^{\Theta} dx\right)^{\frac{1}{\Theta}} &\leq C \left(\fint_{B_{\rho}(y)}|\nabla w_k|^{p-1} dx\right)^{\frac{1}{p-1}}, ~~\Theta>p.
\end{split}
	\end{align}
	As a version of \eqref{eq:bt12} in the ball $B_{10r}(x_3)$, one gives:
	\begin{align}
	\label{eq:es121}
	\begin{split}
\mathcal{L}^n\left(E_{\lambda,\varepsilon}\cap B_r(x)\right) &\le \frac{C}{\left(\varepsilon^{-\frac{1}{\Theta}}\lambda\right)^{\gamma_0}}\left[ \int_{B_{10r}(x_3)}{|\nabla u_k - \nabla w_k|^{\gamma_0}dx} + \right.\\
&\left.+ \int_{B_{10r}(x_3)}{|\nabla u - \nabla u_k|^{\gamma_0}dx}\right] +\frac{C}{\left(\varepsilon^{-\frac{1}{\Theta}}\lambda\right)^{\Theta}}\int_{B_{10r}(x_3)}{|\nabla w_k|^{\Theta}dx}.
	\end{split}
	\end{align}	 
	Since $B_{4r}(x)\subset B_{10r}(x_3)$, similar to \eqref{eq:bt17}, we obtain:
\begin{align}\label{eq:es122}
\begin{split}
\left(\fint_{B_{2r}(x)}|\nabla w_k|^{\Theta} dx\right)^{\frac{1}{\Theta}} &\leq C \left(\fint_{B_{10r}(x_3)}|\nabla w_k|^{p-1} dx\right)^{\frac{1}{p-1}} \\ 
&\le C \left(\fint_{B_{10r}(x_3)}{|\nabla u_k|^{\gamma_0}dx} \right)^{\frac{1}{\gamma_0}}\\ &+ C\left( \fint_{B_{10r}(x_3)}{|\nabla u_k - \nabla w_k|^{\gamma_0}dx} \right)^{\frac{1}{\gamma_0}},
\end{split}
\end{align}
where the second inequality is obtained by using H\"older's inequality and for $\gamma_0>p-1$.
	On the ball $B_{10r}(x_3)$, applying these estimates \eqref{eq:estbound1} and \eqref{eq:estbound2} with \eqref{eq:es122} from above to \eqref{eq:es121}, one obtains the following estimate:
	\begin{align*}
	 &\mathcal{L}^n\left(E_{\lambda,\varepsilon} \cap B_r(x)\right) \\&\le \varepsilon^{\gamma_0\frac{1}{\Theta}}\lambda^{-\gamma_0}r^n \left[C\left( \frac{|\mu_k|(B_{10r}(x_3))}{r^{n-1}}\right)^{\frac{1}{p-1}}\right.\\& \left.+C\frac{|\mu_k|(B_{10r}(x_3))}{r^{n-1}} \left(\fint_{B_{10r}(x_3)}{|\nabla u_k|^{\gamma_0}dx} \right)^{\frac{2-p}{\gamma_0}} \right]^{\gamma_0}	 \\ &+ C\varepsilon^{\gamma_0\frac{1}{\Theta}}\lambda^{-\gamma_0}\int_{B_{10r}(x_3)}{|\nabla u - \nabla u_k|^{\gamma_0}dx} \\
	 &+ C \varepsilon\lambda^{-\Theta}r^n\left[C\left(\fint_{B_{10r}(x_3)}{|\nabla u_k|^{\gamma_0}dx} \right)^{\frac{1}{\gamma_0}} + C\left( \frac{|\mu_k|(B_{10r}(x_3))}{r^{n-1}} \right)^{\frac{1}{p-1}} \right.\\&\left.~~~~+ C\frac{|\mu_k|(B_{10r}(x_3))}{r^{n-1}} \left(\fint_{B_{10r}(x_3)}{|\nabla u_k|^{\gamma_0}dx} \right)^{\frac{2-p}{\gamma_0}} \right]^{\Theta}.
	 \end{align*}
	Letting $k\to \infty$, we can assert that:
	\begin{align*}
&\mathcal{L}^n\left(E_{\lambda,\varepsilon} \cap B_r(x)\right) \\&\le C\varepsilon^{\gamma_0\frac{1}{\Theta}}\lambda^{-\gamma_0}r^n \left[\left( \frac{|\mu|(\overline{B_{10r}(x_3)})}{r^{n-1}}\right)^{\frac{1}{p-1}} \right.\\& \left.+\frac{|\mu|(\overline{B_{10r}(x_3)})}{r^{n-1}} \left(\fint_{B_{10r}(x_3)}{|\nabla u|^{\gamma_0}dx} \right)^{\frac{2-p}{\gamma_0}} \right]^{\gamma_0} \\&+ C\varepsilon\lambda^{-\Theta}r^n\left[\left(\fint_{B_{10r}(x_3)}{|\nabla u|^{\gamma_0}dx} \right)^{\frac{1}{\gamma_0}} + \left( \frac{|\mu|(\overline{B_{10r}(x_3)})}{r^{n-1}} \right)^{\frac{1}{p-1}}\right.\\&\left.~~~~+\frac{|\mu|(\overline{B_{10r}(x_3)})}{r^{n-1}} \left(\fint_{B_{10r}(x_3)}{|\nabla u|^{\gamma_0}dx} \right)^{\frac{2-p}{\gamma_0}} \right]^{\Theta}.
	\end{align*}
	For given $x_1, x_2$ in the previous case and the definition of $x_3$, since $\text{dist}(x,\Omega) \le 4r$, we can easily check that these following bounds:
	\begin{align*}
	\overline{B_{10r}(x_3)} &\subset \overline{B_{14r}(x)}\subset B_{15r}(x_1)\subset D_2,\\
	\overline{B_{10r}(x_3)} &\subset \overline{B_{14r}(x)}\subset B_{15r}(x_2)\subset D_2,
	\end{align*}
	and the following estimates
	\begin{align*}
	\frac{|\mu|(\overline{B_{10r}(x_3)})}{r^{n-1}} &\le \frac{|\mu|(B_{15r}(x_2))}{r^{n-1}} \le 15^{n-1}\mathbf{M}_1(\chi_{D_2}\mu)(x_2)
	\end{align*}
	hold. On the other hand, as $|x_3-x|=\text{dist}(x,\partial\Omega)$, one obtains
	\begin{align}
	\begin{split}
	\left(\fint_{B_{10r}(x_3)}|\nabla u|^{\gamma_0}dx\right)^{\frac{1}{\gamma_0}} &\leq \left( \frac{|B_{15}(0)|}{|B_{10}(0)|} 	\fint_{B_{15r}(x_1)}|\nabla u|^{\gamma_0}dx\right)^{\frac{1}{\gamma_0}}\\&\leq C\left( \sup_{\rho>0} \fint_{B_{\rho}(x_1)}\chi_{D_2}|\nabla u|^{\gamma_0}dx\right)^{\frac{1}{\gamma_0}}
	\\&= C\left(\mathbf{M}\left(\chi_{D_2}|\nabla u|^{\gamma_0}\right)(x_1)\right)^{\frac{1}{\gamma_0}}.
	\end{split}
	\end{align}
	Combining these above estimates together, one finally concludes that 
	$$\mathcal{L}^n\left(E_{\lambda,\varepsilon}\cap B_r(x)\right) \le Cr^n.$$ 
	According to Lemma \ref{lem:mainlem} for $E = E_{\lambda,\varepsilon}$, $F = F_\lambda$, the proof of Theorem \ref{theo:lambda_estimate} is complete and we will refer to this result in the sequel.
\end{proof}\\

Now, let us give a brief proof of Theorem \ref{theo:bst}.

\begin{proof}[Proof of Theorem \ref{theo:bst}]	
Let $0<\rho<T_0$ and $x_0$ be fixed in $\Omega$. We first apply Theorem \ref{theo:lambda_estimate} with $R = \rho$ and the corresponding sets $D_1=B_\rho(x_0)$, $D_2 = B_{10\rho}(x_0)$, there exist $\Theta = \Theta(n,p,\alpha,\beta,c_0)>p$ and a constant $C = C(n,p,\alpha,\beta,c_0,T_0/r_0)>0$ such that the following estimate holds 	
 		\begin{align}
 	&\mathcal{L}^n\left(\left\{({\bf M}(\chi_{B_{10\rho}(x_0)}|\nabla u|^{\gamma_0}))^{1/\gamma_0}>\varepsilon^{-\frac{1}{\Theta}}\lambda, (\mathbf{M}_1(\chi_{B_{10\rho}(x_0)}\mu))^{\frac{1}{p-1}}\le \varepsilon^{\frac{1}{(p-1)\gamma_0}}\lambda \right\}\cap B_{\rho}(x_0) \right) \nonumber\\
 	&\qquad\leq C \varepsilon  \mathcal{L}^n\left(\left\{ ({\bf M}(\chi_{B_{10\rho}(x_0)}|\nabla u|^{\gamma_0}))^{1/\gamma_0}> \lambda\right\}\cap B_{\rho}(x_0)\right),
 	\end{align}
for any $\lambda>\lambda_0(\varepsilon), \varepsilon\in (0,1)$, and for some $\gamma_0 \in \left(\frac{2-p}{2},\frac{(p-1)n}{n-1} \right)$, where $$\lambda_0 = \varepsilon^{-\frac{1}{(p-1)\gamma_0}}\|\nabla u\|_{L^{\gamma_0}(B_{10\rho}(x_0))}\rho^{-\frac{n}{\gamma_0}}.$$
 	Thus, for all $\lambda>\lambda_0(\varepsilon)$ it gives
 	\begin{align}\label{eq:esp}
 	\begin{split}
 	&\mathcal{L}^n\left(\left\{({\bf M}(\chi_{B_{10\rho}(x_0)}|\nabla u|^{\gamma_0}))^{1/\gamma_0}>\varepsilon^{-\frac{1}{\Theta}}\lambda\right\}\cap B_{\rho}(x_0) \right) \\ 
 	&\leq C \varepsilon  \mathcal{L}^n\left(\left\{ ({\bf M}(\chi_{B_{10\rho}(x_0)}|\nabla u|^{\gamma_0}))^{1/\gamma_0}> \lambda\right\}\cap B_{\rho}(x_0)\right)\\
 	&~~~+\mathcal{L}^n\left(\left\{ (\mathbf{M}_1(\chi_{B_{10\rho}(x_0)}\mu))^{\frac{1}{p-1}}> \varepsilon^{\frac{1}{(p-1)\gamma_0}}\lambda \right\}\cap B_{\rho}(x_0) \right).
 	\end{split}
 	\end{align}
 	It is necessary to estimate $({\bf M}(\chi_{B_{10\rho}(x_0)}|\nabla u|^{\gamma_0}))^{1/\gamma_0}$ in $L^{q,s}(B_{\rho}(x_0))$ for $0<q<\Theta$ and $0<s<\infty$. 	Firstly, let us rewrite the Lorentz norm as
 	\begin{align}\label{eq:tp1}
 	\begin{split}
 &	\|({\bf M}(\chi_{B_{10\rho}(x_0)}|\nabla u|^{\gamma_0}))^{1/\gamma_0}\|_{L^{q,s}(B_{\rho}(x_0))}^s\\&=q\int_{0}^{\infty}\lambda^{s-1}\mathcal{L}^n\left(\left\{({\bf M}(\chi_{B_{10\rho}(x_0)}|\nabla u|^{\gamma_0}))^{1/\gamma_0}>\lambda\right\}\cap B_{\rho}(x_0)\right)^{s/q} d\lambda
 \\&=\varepsilon^{-\frac{s}{\Theta}}q\int_{0}^{\infty}\lambda^{s-1}\mathcal{L}^n\left(\left\{({\bf M}(\chi_{B_{10\rho}(x_0)}|\nabla u|^{\gamma_0}))^{1/\gamma_0}>\varepsilon^{-\frac{1}{\Theta}}\lambda\right\}\cap B_{\rho}(x_0)\right)^{s/q}  d\lambda.
 \end{split}
 \end{align}
 Since the estimate \eqref{eq:esp} only holds for all $\lambda>\lambda_0(\varepsilon)$, one splits the integral into the sum of integrals on $(0,\lambda_0)$ and $(\lambda_0, \infty)$ to get
 \begin{align}\label{eq:tp2}
& \int_{0}^{\infty}\lambda^{s-1}\mathcal{L}^n\left(\left\{({\bf M}(\chi_{B_{10\rho}(x_0)}|\nabla u|^{\gamma_0}))^{1/\gamma_0}>\varepsilon^{-\frac{1}{\Theta}}\lambda\right\}\cap B_{\rho}(x_0)\right)^{s/q}  d\lambda \nonumber \\&=\int_{0}^{\lambda_0}\lambda^{s-1}\mathcal{L}^n\left(\left\{({\bf M}(\chi_{B_{10\rho}(x_0)}|\nabla u|^{\gamma_0}))^{1/\gamma_0}>\varepsilon^{-\frac{1}{\Theta}}\lambda\right\}\cap B_{\rho}(x_0)\right)^{s/q}  d\lambda \nonumber \\&~~~+\int_{\lambda_0}^{\infty}\lambda^{s-1}\mathcal{L}^n\left(\left\{({\bf M}(\chi_{B_{10\rho}(x_0)}|\nabla u|^{\gamma_0}))^{1/\gamma_0}>\varepsilon^{-\frac{1}{\Theta}}\lambda\right\}\cap B_{\rho}(x_0)\right)^{s/q}  d\lambda.
 \end{align}
 According to \eqref{eq:tp1} and \eqref{eq:tp2}, it finds:
 \begin{align*}
 &	\|({\bf M}(\chi_{B_{10\rho}(x_0)}|\nabla u|^{\gamma_0}))^{1/\gamma_0}\|_{L^{q,s}(B_{\rho}(x_0))}^s\\
 &\leq C\varepsilon^{-\frac{s}{\Theta}}\lambda_0^{s}\mathcal{L}^n\left( B_{\rho}(x_0)\right)^{s/q}\\
 & ~~~+ C\varepsilon^{-\frac{s}{\Theta}+\frac{s}{q}}\int_{\lambda_0}^{\infty}\lambda^{s-1}\mathcal{L}^n\left(\left\{({\bf M}(\chi_{B_{10\rho}(x_0)}|\nabla u|^{\gamma_0}))^{1/\gamma_0}>\lambda\right\}\cap B_{\rho}(x_0)\right)^{s/q}  d\lambda \\
 & ~~~+ C\varepsilon^{-\frac{s}{\Theta}}\int_{\lambda_0}^{\infty}\lambda^{s-1}\mathcal{L}^n\left(\left\{ (\mathbf{M}_1(\chi_{B_{10\rho}(x_0)}\mu))^{\frac{1}{p-1}}> \varepsilon^{\frac{1}{(p-1)\gamma_0}}\lambda \right\}\cap B_{\rho}(x_0) \right)^{s/q}  d\lambda.
 \end{align*}
 And thus, we check at once that
 	\begin{align*}
 &	\|({\bf M}(\chi_{B_{10\rho}(x_0)}|\nabla u|^{\gamma_0}))^{1/\gamma_0}\|_{L^{q,s}(B_{\rho}(x_0))}^s
 \\
 &\leq C\varepsilon^{-\frac{s}{\Theta}}\lambda_0^{s}\rho^{ns/q}+ C\varepsilon^{-\frac{s}{\Theta}+\frac{s}{q}} 	\|({\bf M}(\chi_{B_{10\rho}(x_0)}|\nabla u|^{\gamma_0}))^{1/\gamma_0}\|_{L^{q,s}(B_{\rho}(x_0))}^s\\
 & ~~~+ C\varepsilon^{-\frac{s}{\Theta}-\frac{s}{(p-1)\gamma_0}}\int_{0}^{\infty}\lambda^{s-1}\mathcal{L}^n\left(\left\{ (\mathbf{M}_1(\chi_{B_{10\rho}(x_0)}\mu))^{\frac{1}{p-1}}> \lambda \right\}\cap B_{\rho}(x_0) \right)^{s/q}  d\lambda
  \\
  &= C\varepsilon^{-\frac{s}{\Theta}}\lambda_0^{s}\rho^{ns/q}+ C\varepsilon^{-\frac{s}{\Theta}+\frac{s}{q}} 	\|{\bf M}(\chi_{B_{10\rho}(x_0)}|\nabla u|^{\gamma_0}))^{1/\gamma_0}\|_{L^{q,s}(B_{\rho}(x))}^s \\
  & ~~~+ C\varepsilon^{-\frac{s}{\Theta}-\frac{s}{(p-1)\gamma_0}}\|(\mathbf{M}_1(\chi_{B_{10\rho}(x_0)}\mu))^{\frac{1}{p-1}}\|_{L^{q,s}(B_{\rho}(x_0))}^s.
 \end{align*}
 Then, it gives us the estimate:
	\begin{align}\label{eq:tp3}
	\begin{split}
&	\|({\bf M}(\chi_{B_{10\rho}(x_0)}|\nabla u|^{\gamma_0}))^{1/\gamma_0}\|_{L^{q,s}(B_{\rho}(x_0))}
\\&\leq  C\varepsilon^{-\frac{1}{\Theta}}\lambda_0\rho^{n/q}+ C\varepsilon^{-\frac{1}{\Theta}+\frac{1}{q}} \|({\bf M}(\chi_{B_{10\rho}(x_0)}|\nabla u|^{\gamma_0}))^{1/\gamma_0}\|_{L^{q,s}(B_{\rho}(x_0))}\\& ~~~+ C\varepsilon^{-\frac{1}{\Theta}-\frac{1}{(p-1)\gamma_0}}\|(\mathbf{M}_1(\chi_{B_{10\rho}(x_0)}\mu))^{\frac{1}{p-1}}\|_{L^{q,s}(B_{\rho}(x_0))}.
\end{split}
\end{align} 
 If $q<\Theta$, let us choose $\varepsilon=\varepsilon_0>0$ such that $C\varepsilon_0^{-\frac{1}{\Theta}+\frac{1}{q}} <1/2$ from \eqref{eq:tp3}, then one obtains
 	\begin{align*}
 &	\|({\bf M}(\chi_{B_{10\rho}(x_0)}|\nabla u|^{\gamma_0}))^{1/\gamma_0}\|_{L^{q,s}(B_{\rho}(x_0))}
 \\&\leq  C\|\nabla u\|_{L^{\gamma_0}(B_{10\rho}(x_0))}\rho^{-n\left(\frac{1}{\gamma_0}-\frac{1}{q} \right)}+ C\|(\mathbf{M}_1(\chi_{B_{10\rho}(x_0)}\mu))^{\frac{1}{p-1}}\|_{L^{q,s}(B_{\rho}(x_0))}.
 \end{align*} 
 The application of Lemma \ref{lem:H:Sep26} enables us to obtain
 \begin{align*}
\|\nabla u\|_{L^{\gamma_0}(B_{10\rho}(x_0))}\le C \rho^{\frac{n}{\gamma_0}+\frac{1-\theta}{p-1}} \|\mathbf{M}_{\theta}^{T_0}(|\mu|)\|_{L^\infty(\Omega)}^{\frac{1}{p-1}}.
 \end{align*}
 In consequence, we get
 	\begin{align*}
 &	\|({\bf M}(\chi_{B_{10\rho}(x_0)}|\nabla u|^{\gamma_0}))^{1/\gamma_0}\|_{L^{q,s}(B_\rho(x_0))}
 \\&\leq  C \rho^{-n \left(\frac{1}{\gamma_0}-\frac{1}{q}\right)} \rho^{\frac{n}{\gamma_0}+\frac{1-\theta}{p-1}} \|\mathbf{M}_{\theta}^{T_0}(|\mu|)\|_{L^\infty(\Omega)}^{\frac{1}{p-1}}+ C\|(\mathbf{M}_1(\chi_{B_{10\rho}(x_0)}\mu))^{\frac{1}{p-1}}\|_{L^{q,s}(B_\rho(x_0))}
  \\&=  C\rho^{-\frac{n}{\gamma_0}+\frac{n}{q}+\frac{n}{\gamma_0}+\frac{1-\theta}{p-1}}  \|\mathbf{M}_{\theta}^{T_0}(|\mu|)\|_{L^\infty(\Omega)}^{\frac{1}{p-1}}+ C\|(\mathbf{M}_1(\chi_{B_{10\rho}(x_0)}\mu))^{\frac{1}{p-1}}\|_{L^{q,s}(B_\rho(x_0))}
 \end{align*} 
 that yields
 \begin{align*}
 &	\rho^{-\frac{n}{q}+\frac{\theta-1}{p-1}} \|({\bf M}(\chi_{B_{10\rho}(x_0)}|\nabla u|^{\gamma_0}))^{1/\gamma_0}\|_{L^{q,s}(B_\rho(x_0))}
 \\&\leq  C \|\mathbf{M}_{\theta}^{T_0}(|\mu|)\|_{L^\infty(\Omega)}^{\frac{1}{p-1}}+ C\rho^{-\frac{n}{q}+\frac{\theta-1}{p-1}}\|(\mathbf{M}_1(\chi_{B_{10\rho}(x_0)}\mu))^{\frac{1}{p-1}}\|_{L^{q,s}(B_\rho(x_0))}.
 \end{align*} 
 Finally, by taking the supremum for all $\rho \in (0,T_0)$ and $x_0 \in \Omega$, we conclude the desired result:
	\begin{align*}
&	\sup_{\rho \in (0,T_0),x_0 \in \Omega}\rho^{-\frac{n}{q}+\frac{\theta-1}{p-1}} \|({\bf M}(\chi_{B_{10\rho}(x_0)}|\nabla u|^{\gamma_0}))^{1/\gamma_0}\|_{L^{q,s}(B_\rho(x_0))}
\\&\leq  C \|\mathbf{M}_{\theta}^{T_0}(|\mu|)\|_{L^\infty(\Omega)}^{\frac{1}{p-1}}+ \sup_{\rho \in (0,T_0),x_0 \in \Omega}C\rho^{-\frac{n}{q}+\frac{\theta-1}{p-1}}\|(\mathbf{M}_1(\chi_{B_{10\rho}(x_0)}\mu))^{\frac{1}{p-1}}\|_{L^{q,s}(B_\rho(x_0))}.
\end{align*} 
\end{proof}

\begin{proof}[Proof of Theorem \ref{theo:P}]	

Let $0<\rho<T_0$ and $x_0$ be fixed in $\Omega$. From what has already been proved in Theorem \ref{theo:bst} and the definition of Lorentz-Morrey norm \eqref{eq:LMsp} it gets
\begin{align*}
&\|({\bf M}(\chi_{B_{10\rho}(x_0)}|\nabla u|^{\gamma_0}))^{1/\gamma_0}\|_{L^{q,s;\frac{q(\theta-1)}{p-1}}(B_\rho(x_0))}
 \\&\leq  C \|\mathbf{M}_{\theta}^{T_0}(|\mu|)\|_{L^\infty(\Omega)}^{\frac{1}{p-1}}+ \sup_{\rho \in (0,T_0),x_0 \in \Omega}C\rho^{-\frac{n}{q}+\frac{\theta-1}{p-1}}\|(\mathbf{M}_1(\chi_{B_{10\rho}(x_0)}\mu))^{\frac{1}{p-1}}\|_{L^{q,s}(B_\rho(x_0))},
 \end{align*}
that yields
\begin{align}\label{eq:CM}
\begin{split}
\|\nabla u\|_{L^{q,s;\frac{q(\theta-1)}{p-1}}(\Omega)} &\leq C \|\mathbf{M}_{\theta}^{T_0}(|\mu|)\|_{L^\infty(\Omega)}^{\frac{1}{p-1}}  \\ &+C \sup_{\rho \in (0,T_0),x_0 \in \Omega}\rho^{-\frac{n}{q}+\frac{\theta-1}{p-1}}\|(\mathbf{M}_1(\chi_{B_{10\rho}(x_0)}\mu))^{\frac{1}{p-1}}\|_{L^{q,s}(B_\rho(x_0))}.
\end{split}
\end{align}
In order to get the gradient estimate of solution in Lorentz-Morrey spaces, it is sufficient to show that each term on the right side of \eqref{eq:CM} is bounded by the Lorentz-Morrey norm of measure data $\mu$. In particular, it is sufficient to show that
\begin{align}
\label{eq:CM1}
\|\mathbf{M}_{\theta}^{T_0}(|\mu|)\|_{L^\infty(\Omega)} \le C_1 \||\mu|^{\frac{1}{p-1}}\|_{L^{\frac{q(\theta-1)}{\theta}, \frac{s(\theta-1)}{\theta};\frac{q(\theta-1)}{p-1}}(\Omega)}^{p-1}
\end{align}
and
\begin{align}
\label{eq:CM2}
\begin{split}
\sup_{\rho \in (0,T_0),x_0 \in \Omega} &\rho^{-\frac{n(p-1)}{q}+\theta-1}\|(\mathbf{M}_1(\chi_{B_{10\rho}(x_0)}\mu))\|_{L^{\frac{q}{p-1},\frac{s}{p-1}}(B_\rho(x_0))}  \\ 
& \quad ~~~\le C_2 \||\mu|^{\frac{1}{p-1}}\|_{L^{\frac{q(\theta-1)}{\theta}, \frac{s(\theta-1)}{\theta};\frac{q(\theta-1)}{p-1}}(\Omega)}^{p-1}.
\end{split}
\end{align}
hold. First, we can proceed to the proof of \eqref{eq:CM1}. For $0<\rho<T_0$ and $x_0 \in \Omega$ we have
\begin{align*}
\||\mu|^{\frac{1}{p-1}}\|_{L^{\frac{q(\theta-1)}{\theta}, \frac{s(\theta-1)}{\theta};\frac{q(\theta-1)}{p-1}}(\Omega)}^{p-1} & = \|\mu\|_{L^{\frac{q(\theta-1)}{\theta(p-1)}, \frac{s(\theta-1)}{\theta(p-1)};\frac{q(\theta-1)}{p-1}}(\Omega)} \\
&\ge \|\mu\|_{L^{\frac{q(\theta-1)}{\theta(p-1)}, \frac{s(\theta-1)}{\theta(p-1)};\frac{q(\theta-1)}{p-1}}(B_{10\rho}(x_0))} \\ 
&\ge\|\mu\|_{L^{\frac{q(\theta-1)}{\theta(p-1)}, \infty;\frac{q(\theta-1)}{p-1}}(B_{10\rho}(x_0))}\\
&\ge C\rho^{\frac{\frac{q(\theta-1)}{p-1}-n}{\frac{q(\theta-1)}{\theta(p-1)}}} \left[\mathcal{L}^n\left(B_{10\rho}(x_0) \right) \right]^{-1+\frac{1}{\frac{q(\theta-1)}{\theta(p-1)}}} |\mu|\left(B_{10\rho}(x_0) \right)\\
&= C \rho^{\theta-\frac{n\theta(p-1)}{q(\theta-1)}} \rho^{-n+\frac{n}{\frac{q(\theta-1)}{\theta(p-1)}}}|\mu|\left(B_{10\rho}(x_0) \right) \\ 
&= C\frac{|\mu|\left(B_{10\rho}(x_0) \right)}{\rho^{n-\theta}},
\end{align*}
that leads to our desired result in \eqref{eq:CM1} by taking the supremum both side for all $0<\rho<T_0$ and $x_0 \in \Omega$, where it follows the definition of $\mathbf{M}_{\theta}^{T_0}(|\mu|)$ in \eqref{eq:MaT}.

On the other hand, one refers to \cite[Theorem 1.1]{55Ph1} to get that for any $x \in 10B_0 = B_{10\rho}(x_0)$ one obtains:
\begin{align}\label{eq:bs1}
\mathbf{M}_1\left(\chi_{10B_0}|\mu| \right)(x)  \le C\left[\mathbf{M}\left( \chi_{10B_0}|\mu|\right)(x) \right]^{1-\frac{1}{\theta}} \|\mu\|^{\frac{1}{\theta}}_{L^{\frac{q(\theta-1)}{\theta(p-1)}, \frac{s(\theta-1)}{\theta(p-1)};\frac{q(\theta-1)}{p-1}}(10B_0)},
\end{align}
and this can be used to give the estimate in \eqref{eq:CM2}. Indeed, 
\begin{align*}
& \|\mathbf{M}_1\left(\chi_{10B_0}|\mu| \right)\|_{L^{\frac{q}{p-1},\frac{s}{p-1}}(10B_0)} \\
& \quad\le C\|[\mathbf{M}\left(\chi_{10B_0}|\mu| \right)]^{\frac{\theta-1}{\theta}}\|_{L^{\frac{q}{p-1},\frac{s}{p-1}}(10B_0)} \|\mu\|^{\frac{1}{\theta}}_{L^{\frac{q(\theta-1)}{\theta(p-1)}, \frac{s(\theta-1)}{\theta(p-1)};\frac{q(\theta-1)}{p-1}}(10B_0)}\\
& \quad \le C\|\mathbf{M}\left(\chi_{10B_0}|\mu| \right)\|^{\frac{\theta-1}{\theta}}_{L^{\frac{q(\theta-1)}{\theta(p-1)}, \frac{s(\theta-1)}{\theta(p-1)}}(10B_0)}\|\mu\|^{\frac{1}{\theta}}_{L^{\frac{q(\theta-1)}{\theta(p-1)}, \frac{s(\theta-1)}{\theta(p-1)};\frac{q(\theta-1)}{p-1}}(10B_0)}.
\end{align*}
According to the boundedness property of maximal function $\mathbf{M}$ in Remark \ref{rem:boundMlorentz}, it finds:
\begin{align*}
\begin{split}
 \|\mathbf{M}_1\left(\chi_{10B_0}|\mu| \right)\|&_{L^{\frac{q}{p-1},\frac{s}{p-1}}(10B_0)} \\
& \quad\le C \|\mu\|_{L^{\frac{q(\theta-1)}{\theta(p-1)}, \frac{s(\theta-1)}{\theta(p-1)}}(10B_0)}^{\frac{\theta-1}{\theta}}\|\mu\|^{\frac{1}{\theta}}_{L^{\frac{q(\theta-1)}{\theta(p-1)}, \frac{s(\theta-1)}{\theta(p-1)};\frac{q(\theta-1)}{p-1}}(10B_0)},
\end{split}
\end{align*}
and from the definition of Lorentz-Morrey norm \eqref{eq:LMsp}, it follows that:
\begin{align}\label{eq:bs2}
\begin{split}
 \|\mathbf{M}_1\left(\chi_{10B_0}|\mu| \right)\|_{L^{\frac{q}{p-1},\frac{s}{p-1}}(10B_0)}
& \le C \rho^{\frac{n-\frac{q(\theta-1)}{p-1}}{\frac{q(\theta-1)}{\theta(p-1)}} . \frac{\theta-1}{\theta}} \|\mu\|_{L^{\frac{q(\theta-1)}{\theta(p-1)}, \frac{s(\theta-1)}{\theta(p-1)};\frac{q(\theta-1)}{p-1}}(10B_0)}\\
&  \le C \rho^{\frac{n(p-1)}{q}-\theta+1}  \|\mu\|_{L^{\frac{q(\theta-1)}{\theta(p-1)}, \frac{s(\theta-1)}{\theta(p-1)};\frac{q(\theta-1)}{p-1}}(\Omega)}.
\end{split}
\end{align}
Multiplying both sides of \eqref{eq:bs2} by $\rho^{\frac{-n(p-1)}{q}+\theta-1}$, we turn to
\begin{align*}
\begin{split}
\rho^{\frac{-n(p-1)}{q}+\theta-1}  \|\mathbf{M}_1\left(\chi_{10B_0}|\mu| \right)\|_{L^{\frac{q}{p-1},\frac{s}{p-1}}(10B_0)} 
& \le C \|\mu\|_{L^{\frac{q(\theta-1)}{\theta(p-1)}, \frac{s(\theta-1)}{\theta(p-1)};\frac{q(\theta-1)}{p-1}}(\Omega)}\\
&  = C \||\mu|^{\frac{1}{p-1}}\|_{L^{\frac{q(\theta-1)}{\theta}, \frac{s(\theta-1)}{\theta};\frac{q(\theta-1)}{p-1}}(\Omega)}^{p-1},
\end{split}
\end{align*}
and the proof of \eqref{eq:CM2} is complete.
\end{proof}

\subsection*{Acknowledgments} The author T.N. Nguyen was supported by Ho Chi Minh City University of Education under grant No. B2017-SPS-12.

\end{document}